\documentclass[11pt,twoside,reqno]{amsart}
\usepackage{amsmath}

\usepackage{amsfonts}
\usepackage{amsthm}
\usepackage{amssymb}
\usepackage{cite}

\usepackage{times}

\makeatletter 
\def\Links{\tagsleft@true}\def\Rechts{\tagsleft@false}
\makeatother

\newcommand{\dhr}{\mathrel{\lhook\joinrel\relbar\kern-.8ex\joinrel\lhook\joinrel\rightarrow}}

\newcommand{\Wqb}{W_{q,\mathcal{B}}}
\newcommand{\Wpb}{W_{p,\mathcal{B}}}

\newcommand{\R}{\mathbb{R}}
\newcommand{\C}{\mathbb{C}}

\newcommand{\N}{\mathbb{N}}

\newcommand{\A}{\mathbb{A}}

\newcommand{\mS}{\mathbb{S}}

\newcommand{\ml}{\mathcal{L}}

\newcommand{\Om}{\Omega}

\newcommand{\ve}{\varepsilon}
\newcommand{\rd}{\mathrm{d}}

\newcommand{\bqn}{\begin{equation}}
\newcommand{\eqn}{\end{equation}}
\newcommand{\bqnn}{\begin{equation*}}
\newcommand{\eqnn}{\end{equation*}}
\newcommand{\bear}{\begin{eqnarray}}
\newcommand{\eear}{\end{eqnarray}}
\newcommand{\bean}{\begin{eqnarray*}}
\newcommand{\eean}{\end{eqnarray*}}

\newcommand{\bbWp}{\mathbb{W}_p(E_0,E_1)}
\newcommand{\bbWpp}{\mathbb{W}_p^+(E_0,E_1)}
\newcommand{\bbWppd}{\dot{\mathbb{W}}_p^+(E_0,E_1)}

\newtheorem{theorem}{Theorem}[section]
\newtheorem{corollary}[theorem]{Corollary}
\newtheorem{lemma}[theorem]{Lemma}
\newtheorem{proposition}[theorem]{Proposition}

\newtheorem{remark}[theorem]{Remark}
\numberwithin{equation}{section}
\begin{document}
%
\title{Some Results Based on Maximal Regularity Regarding Population Models with Age and Spatial Structure}

\author{Christoph Walker}
\address{Leibniz Universit\"at Hannover, Institut f\" ur Angewandte Mathematik, Welfengarten 1, D--30167 Hannover, Germany} 
\email{walker@ifam.uni-hannover.de}
%
%
\begin{abstract}
We review some results on abstract linear and nonlinear population models with age and spatial structure. The results are mainly based on the assumption of maximal $L_p$-regularity of the spatial dispersion term. In particular, this property allows us to characterize completely the generator of the underlying linear semigroup and to give a simple proof of asynchronous exponential  growth of the semigroup. Moreover, maximal regularity is also a powerful tool in order to establish the existence of nontrivial positive equilibrium solutions to nonlinear equations by fixed point arguments or bifurcation techniques. We illustrate the results with examples.
\end{abstract}
%
\maketitle
\pagestyle{myheadings}
\markboth{{\sc Ch. Walker}}{\sc{Population Models with Age and Spatial Structure}}
%
%
\section{Introduction} 

The mathematical description of the dynamics of populations has a long history dating  back to the deterministic models of Euler (1760) and Malthus (1798) for exponential growth and Verhulst (1838) for logistic growth. More advanced models distinguish the individuals of a population by means of a certain characteristic  as e.g. age or spatial position as such may have an influence on the dynamics of the population. Indeed, individuals of different age or spatial position may respond differently, for example, to external effects and likewise are subject to different death and birth processes as well as slower or faster spatial dispersion.

The early linear models introduced by Sharpe $\&$ Lotka (1911) \cite{SharpeLotka11}, McKendrick (1926) \cite{McKendrick} and von Foerster (1959) \cite{vFoerster} and the later  nonlinear model by Gurtin $\&$ MacCamy (1974) \cite{GurtinMacCamy74} are the basis for the description of the evolution of age-structured populations by means of partial differential equations (see \eqref{E3}  below). 
Spatial dispersion of individuals was taken into account within linear theory by Gurtin (1973) \cite{Gurtin73} and in a nonlinear model by Gurtin $\&$ MacCamy (1981) \cite{GurtinMacCamy81}. Since the introduction of these basic models, the research on population  dynamics has grown tremendously and is still very active. Applications of models incorporating both age and spatial structure include epidemic models (e.g. \cite{BusenbergLanglais, 54, 120}), bacteria swarming (e.g. \cite{14, ES, LaurencotWalker_Proteus}), tumor invasion (e.g. \cite{16, DW1, DW2, WalkerDIE, WalkerEJAM}) and many more (see  for instance \cite{WebbSpringer} and the references therein).  In this text we shall focus our attention on a particular, but prototypical model for the dynamics of an age-structured population of one-sex individuals subject to spatial diffusion.

\subsection{Spatially Homogeneous Age-Structured Populations}

For the investigation of spatially inhomogeneous populations it is instructive to briefly recall the classical (linear) model of  Sharpe $\&$ Lotka and McKend\-rick for purely age-structured populations. Many of the  results for this case can then later be carried over to population models with additional spatial structure within a suitable functional analytic framework. For more on age-structured population models we refer e.g. to \cite{IannelliMartchevaMilner,ThiemeBook,WebbSpringer,WebbBook}.

Let $u=u(t,a)\ge0$ be the age-density function of a population at time $t\ge 0$ and chronological age $a\in [0,a_m)$ with $a_m\in (0,\infty]$ denoting the maximal\footnote{If $a_m$ is finite, individuals may attain age greater than $a_m$ but are no longer tracked in the model.} age. If $\mu=\mu(a)\ge 0$ is the age-specific per capita {\it death rate}, then  
$$
u(t+h, a+h)-u(t,a)=-\mu(a) u(t,a) h
$$
with time increment $h$, which entails a balance law for the directional derivative:
\begin{equation*}
\lim_{h\to 0} \frac{u(t+h, a+h)-u(t,a)}{h}=-\mu(a) u(t,a)\,, 
\end{equation*}
or, if the partial derivatives exist,
$$
\partial_t u(t,a)+\partial_a u(t,a)=-\mu(a) u(t,a)\,. 
$$
Letting  $\beta=\beta(a)\ge 0$ be the per capita {\it birth rate}, the {\it total birth rate} at time $t$ is
$$
B(t):=\int_0^{a_m}\beta(a) u(t,a)\,\rd a\,,
$$
which gives the age boundary condition
$$
u(t,0)=\int_0^{a_m}\beta(a) u(t,a)\,\rd a\,. 
$$
Consequently, the evolution of the population is governed by 
\begin{subequations}\label{E3}
\begin{align}
\partial_t u +\partial_a u&=-\mu(a) u\ ,&& t>0\ ,\ & a\in (0,a_m)\, ,\label{3}\\
u(t,0)&=\int_0^{a_m}\beta(a) u(t,a)\,\rd a\ ,&& t>0\, ,\ & \label{4}\\
u(0,a)&=\phi(a)\ , && &a\in (0,a_m)\,,\label{5}
\end{align}
\end{subequations}
with initial distribution $\phi$. Integrating \eqref{3} along characteristics yields the solution in the form
\bqn\label{6}
u(t,a)=\left\{\begin{array}{ll} \Pi(a,a-t)\phi(a-t)\,, & 0\le t\le a<a_m\,,\\
\Pi(a,0) B(t-a)\, , &0\le a<a_m\,,\quad t>a\,,\\
 \end{array} \right.
\eqn
where
$$
\Pi(a,\sigma):=e^{-\int_\sigma^a\mu(\tau)\,\rd \tau}\, ,\quad 0\le \sigma < a\ ,
$$
can be interpreted as the probability\footnote{More precisely, if $\mu(a)>0$ for $a\in (0,a_m)$ and $\int_0^{a_m} \mu(a)\rd a=\infty$, then $1-\Pi(a,0)$ is a probability distribution with density $\mu(a)\Pi(a,0)$.} that an individual of age $\sigma$ survives to age $a$. According to \eqref{4}, the total birth rate $B(t)=u(t,0)$ satisfies the Volterra equation
\bqn\label{8}
B(t)=\int_0^t h(a)\beta(a)\Pi(a,0)B(t-a)\,\rd a+\int_t^{a_m} h(a)\beta(a)\Pi(a,a-t)\phi(a-t)\,\rd a
\eqn
for $t\ge 0$ with cut-off function $h(a):=1$ if $a\in(0,a_m)$ and $h(a):=0$ otherwise,  which is also known as {\it renewal equation} (e.g, see \cite{Feller41, ThiemeBook}).
Obviously, predictions on the asymptotic behavior of solutions are of great interest. 
In this context, {\it stable age distributions} (or: {\it persistence solutions}) are of importance, that is, solutions to \eqref{E3} with separable variables of the form 
$$
u(t,a)=v(t)w(a)\,.
$$ 
Plugging the form of a stable age distribution into \eqref{3} gives
\bqn\label{12}
u(t,a)= e^{\lambda_0 (t-a)}\Pi(a,0)w(0)\ ,\quad t\ge 0,\quad a\in (0,a_m)\ ,
\eqn
with $w(0)> 0$, and where the  {\it Malthusian parameter} (or: {\it intrinsic growth rate}) $\lambda_0$ is a real parameter determined from \eqref{4} by the {\it characteristic equation}
\begin{equation}\label{chareq}
r(\lambda_0)=1
\end{equation}
with
$$
r(\lambda):=\int_0^{a_m}e^{-\lambda a}\beta(a) \Pi(a,0)\,\rd a\,.
$$
Owing to the monotonicity property of $r$, it is clear that there is a unique real value $\lambda_0$ satisfying  \eqref{chareq} and that all complex solutions to \eqref{chareq} come in complex conjugate pairs with real parts smaller than $\lambda_0$. That is, $\lambda_0$ can be seen as a growth bound.

The celebrated {\it  renewal theorem} \cite{SharpeLotka11} (for rigorous proofs see\cite{Feller41, ThiemeBook, WebbBook}) states that -- under some suitable technical assumption on  $\mu$, $\beta$, and $\phi$ -- any solution $u$ to \eqref{E3} approaches a stable age distribution, that is,
\begin{equation}\label{renewal}
\lim_{t\rightarrow\infty}\int_0^{a_m}\left\vert e^{-\lambda_0 t}u(t,a)- e^{-\lambda_0 a}\Pi(a,0)P_{\lambda_0}(\phi)\right\vert\,\rd a=0\ ,
\end{equation}
where the number $P_{\lambda_0}(\phi)$ is given as
$$
P_{\lambda_0}(\phi)=\int_0^{a_m}\beta(a)  \int_0^ae^{-\lambda_0(a-\sigma)}\Pi(a,\sigma)\phi(\sigma)\,\rd\sigma\,\rd a \left(\int_0^{a_m}\beta(a) e^{-\lambda_0 a}\Pi(a,0) a\,\rd a \right)^{-1}\ .
$$
We shall see a similar result for spatially structured models (see Section~\ref{Sec2.2}).\\

According to \eqref{12}, the condition $\lambda_0=0$ determines the equilibrium (i.e. time-independent) solutions $u=u(a)\ge 0$ to \eqref{E3} which are of the form
\bqn\label{EE0}
 u(a)=\Pi(a,0) u(0)\ ,\quad a\ge 0\,,
\eqn
with $u(0)\ge 0$.
Obviously, $u\equiv 0$ is always an equilibrium. Due to \eqref{chareq}, nontrivial equilibria exist if and only if the condition
\bqn\label{9}
r(0)=1
\eqn
is satisfied, where  the {\it net reproduction number}
\bqn\label{10}
r(0)=\int_0^{a_m} \beta(a)\Pi(a,0)\,\rd a
\eqn
gives the average number of newborns of an individual over its lifespan. Thus, nontrivial equilibria for the {\it linear} model \eqref{E3} only exist in the very special case of exact reproduction.

In general, we infer from \eqref{chareq} and the monotonicity of $r$ that
\begin{equation}\label{xx}
\begin{split}
\mathrm{sign}(r(0)-1) = \mathrm{sign}(\lambda_0)\,.
\end{split}
\end{equation} 
In particular, \eqref{renewal} implies that if  $r(0)<1$, then the trivial equilibrium $u\equiv 0$  is globally asymptotically stable in the phase space $L_1((0,a_m))$ while $r(0)>1$ yields an {\it asynchronous exponential growth} of the population (see e.g. \cite{WebbSpringer,GyllenbergWebb}).

Of course, the investigation of linear models is of utmost importance for a deeper understanding of population dynamics. However, a drawback of the linear theory is that an equilibrium exists if and only if the restrictive condition \eqref{9} holds. Nonlinear population models better cope with the intuitive expectation that equilibria should exist in many populations. The model introduced by Gurtin $\&$ MacCamy
 \cite{GurtinMacCamy74} involves birth and death rates $\beta=\beta(U,a)$ and $\mu=\mu(U,a)$ depending on the {\it total population}
$$
U:=\int_0^{a_m} u(a)\rd a\,.
$$
In this case there is a nontrivial equilibrium if (and only if) the equation analogue to \eqref{9},
$$
r_U:=\int_0^{a_m} \beta(U,a)\Pi_U(a,0)\,\rd a=1
$$
posses a solution $U>0$, where 
\bqnn
\Pi_U(a,b):=e^{-\int_b^a\mu(U,\sigma)\,\rd \sigma}\ ,\quad 0\le b < a\,.
\eqnn
The equilibrium $u$ is then given by
$$
u(a)=\Pi_U(a,0)\frac{U}{\int_0^{a_m} \Pi_U(a,0)\,\rd a}\ ,\quad a\in [0,a_m)\ .
$$
In many realistic applications, the death rate is an increasing function of the total population while the birth rate is decreasing with respect to this parameter. In particular, $r_U\le r_0$ for $U\ge 0$. Thus, for the existence of nontrivial equilibria in this case, the condition $r_0\ge 1$ is necessary and  $r_0>1$ can be shown to be sufficient \cite{Pruess83}.

\subsection{Age- and Spatially Structured Model}

We now consider a simple prototype model for an age- and spatially structured population by adding a spatial dependence to the equations \eqref{E3}.
The distribution density of the population with respect to age $a\in [0,a_m)$ and spatial position $x\in\Om$ at time $t\ge 0$ is denoted by $u=u(t,a,x)\ge 0$, where still $a_m\in (0,\infty]$ is the maximal age and $\Om$ is a smooth subset of $\R^n$ representing the habitat of the population. Suppose that the individuals' movement can be described by a diffusion term $\mathrm{div}(d(a,x)\nabla_xu)$, where the dispersal speed $d(a,x)>0$ is age-specific and takes into account spatial heterogeneity of the environment. 
Also the birth and death rates, $\beta=\beta(a,x)\ge 0$ and $\mu=\mu(a,x)\ge 0$, respectively, 
may depend on age $a\in (0,a_m)$ and spatial position $x\in\Om$. If the individuals are bound to the habitat $\Omega$, then the dynamics of the population with initial distribution $\phi=\phi(a,x)\ge 0$ is governed by the equations
\begin{subequations}\label{Eu1a}
\begin{align}
\partial_t u+\partial_a u&=\mathrm{div}\big(d(a,x)\nabla_xu\big)-\mu(a,x)u\ , && t>0\, , &  a\in (0,a_m)\, ,& & x\in\Om\ ,\label{u1a}\\
u(t,0,x)&=\int_0^{a_m} \beta(a,x)u(t,a,x)\,\rd a\ ,& & t>0\, , & & & x\in\Om\ ,\label{u2a}\\
\partial_\nu u(t,a,x)&=0\ ,& & t>0\, , &  a\in (0,a_m)\, ,& & x\in\partial\Om\ ,\label{u3a}\\
u(0,a,x)&=\phi(a,x)\ ,& & &  a\in (0,a_m)\, , & & x\in\Om\ ,\label{u4a}
\end{align}
\end{subequations}
with $\nu$ denoting the outward unit normal on $\partial \Om$. Of course, as in the case of spatially homogeneous populations, more realistic models are nonlinear and include population-dependent spatial movement and vital rates. We shall consider such models later on. 
Equations of the form \eqref{Eu1a} and various variants thereof have attracted considerable interest in the last decades, for example see \cite{14,16,BusenbergLanglais,DucroutMagal,Huyer,Rhandi,RhandiSchnaubelt_DCDS99,ThiemeDCDS,WalkerDIE,WalkerDCDSA10,WebbSpringer,DelgadoEtAl2,GuoChan,KunyaOizumi,LaurencotWalker_Proteus,MagalThieme,WalkerEJAM,120,54,DS2,DS3,DS4,DW1,DW2} and the references therein though this list is far from being complete. \\

For the understanding of what follows it is worth pointing out that the equations \eqref{Eu1a} fit into a more abstract and more general setting.
For this we put
$$
A(a)w:=-\mathrm{div}\big(d(a,\cdot)\nabla_xw\big)+\mu(a,\cdot)w\ ,\quad w\in E_1\ \ ,
$$
where e.g. $E_1:=\Wqb^2(\Om)$ denotes  the space consisting of all  functions $w:\Om\rightarrow \R$ of the Sobolev space  $W_q^{2}(\Om)$ with $q\in (1,\infty)$ satisfying the boundary condition $\partial_\nu w=0$ on $\partial\Om$. Observe that the operator $A$ combines the information on spatial dispersion and mortality.

With this notation and by dropping the $x$-dependence for simplicity, \eqref{Eu1a} can be reformulated as an abstract problem  of the form
\begin{subequations}\label{E1a}
\begin{align}
&\partial_t u\,+\, \partial_au \, +\,     A(a)\,u\, =0\ , && t>0\, ,\quad a\in (0,a_m)\ ,\label{1a}\\ 
&u(t,0)\, =\, \int_0^{a_m}\beta(a)\, u(t,a)\, \rd a\ ,& & t>0\, ,  \label{2a}\\
&u(0,a)\, =\,  \phi(a)\ , & &\hphantom{t>0\ ,\quad} a\in (0,a_m)\,,\label{3a}
\end{align}
\end{subequations}
where $A(a)$ is an unbounded operator in $E_0:=L_q(\Omega)$ with domain $E_1$. 
In the following, we shall focus our attention on this abstract form \eqref{E1a}, where the $x$-dependence is ``hidden'' in the functional analytic framework. This, of course, yields a much wider range of applications and the inclusion of more general elliptic operators subject to possibly other boundary conditions.\\

It is a very striking feature of population models with age and spatial structure that the equilibrium (i.e. time-independent) version of \eqref{E1a}  leads to a ``parabolic'' equation
\begin{align*}
& \partial_au \, +\,     A(a)u\, =0\ , \qquad a\in (0,a_m)\, ,\\ 
&u(0)\, =\, \int_0^{a_m}\beta(a)\, u(a)\, \rd a\, ,   
\end{align*}
for $u=u(a)$. Note well, however, that this is {\it not} an evolution equation in the usual sense since the initial state does not determine the future in view of the nonlocal initial condition. Nevertheless, many techniques from the theory of evolution equations are very useful for a thorough investigation of such equations. In particular, it turns out that the assumption of {\it maximal $L_p$-regularity} of the underlying operator $A$ yields a very powerful tool. Roughly speaking, it provides for any $f\in L_p((0,a_m),E_0)$ a unique solution $u$ to the Cauchy problem
$$
\partial_au + A(a)u =f(a)\,,\quad a\in (0,a_m)\,,\qquad u(0)=0\,,
$$ 
with optimal regularity in the sense that each term $\partial_a u$ and $A u$ belongs again to $L_p((0,a_m),E_0)$ (see assumption $(A_3)$ below for a more precise statement).
The results we are presenting herein are based on this assumption of maximal regularity, which -- as will be explained later in Section~\ref{Sec4} in more detail -- is not really restrictive in applications. In Section~\ref{Sec2.1} we shall see that a  semigroup can be associated with the time-dependent problem \eqref{E1a} whose generator can be characterized completely due the maximal regularity assumption. We shall also see then in Section~\ref{Sec2.2}  that maximal regularity allows us to give a rather simple proof for asynchronous exponential  growth of the semigroup. Furthermore, we rely on the maximal regularity property to establish the existence of nontrivial positive equilibrium solutions for the nonlinear variant of \eqref{E1a} with population-dependent operator $A=A(u,a)$ and population-dependent birth modulus $\beta=\beta(u,a)$: in Section~\ref{Sec3.1} we use a fixed point method in conical shells and in Section~\ref{Sec3.2} we use unilateral local and global bifurcation techniques. 
Some of the proofs given herein are simpler than in the original papers and are included for illustrative purposes. Finally, we give in Section~\ref{Sec4} two applications of the bifurcation results.\\

\subsection*{General Assumptions and Notations} In the following we rely on the theory of semigroups and evolution operators and often use basic properties implicitly. For more on this we refer e.g. to the monographs \cite{LQPP,EngelNagel,EngelNagel2,BatkaiFijavzRhandi}.

If $E$ and $F$ are two Banach spaces (or locally convex spaces) we write $\ml(E,F)$ for the bounded and linear operators from $E$ into $F$ and  $\mathcal{K}(E,F)$ for such operators which are also compact. We set $\ml(E):=\ml(E,E)$ and $\mathcal{K}(E):=\mathcal{K}(E,E)$. If $E$ and $F$ are ordered Banach spaces, we denote the positive operators from $E$ into $F$ by $\ml_+(E,F)$.

From now on and throughout this text we shall make the following assumptions. In order to have a simpler exposition of the results in Section~\ref{Sec2} we only consider the case $a_m<\infty$ but remark that the main results presented herein remain true in the case $a_m=\infty$ (see Remark~\ref{R0}). We set $J:=[0,a_m]$ and let $E_0$ denote a real Banach space ordered by closed convex cone $E_0^+$ (in the following we do not distinguish $E_0$ from its complexification required at certain points). We let $E_1$ be a dense  subspace of $E_0$ such that 
$E_1\dhr E_0$, 
that is, $E_1$ is continuously and compactly embedded in $E_0$. We set $E_\theta:= (E_0,E_1)_\theta$ for
 $\theta\in (0,1)$, where $(\cdot,\cdot)_\theta$ is an arbitrarily fixed admissible interpolation functor. Then  
\begin{equation}\label{dhr}
E_1\dhr E_\theta\dhr E_0\,,\quad \theta\in (0,1)\,,
\end{equation} 
 and $E_\theta$ is equipped with the order naturally induced by $E_0^+$. Given $p\in [1,\infty)$ we use the notation{\footnote{We shall suppress the age interval $J$ also in the writing of other function spaces as no confusion seems likely.}
$$
L_p(E_\theta):=L_p(J,E_\theta)
$$
with positive cone $L_p^+(E_\theta):=L_p^+(J,E_\theta)$.
Further assumptions will be mentioned in the text explicitly.

\section{Linear Theory for Age- and Spatially Structured Populations}\label{Sec2}

In this section we focus  on the linear abstract problem \eqref{E1a}.
It was shown in \cite{WebbSpringer} that a strongly continuous semigroup on $L_1(J,E_0)$ can be associated with \eqref{E1a} if $-A$ is independent of age and generates itself a strongly continuous semigroup on $E_0$ (see also e.g. \cite{GuoChan, Huyer, WalkerDIE}). As in the spatially homogeneous setting \eqref{E3}, this is derived by formally integrating along characteristics what gives the semigroup rather explicitly. Based on \cite{WalkerMOFM} we recall in Section~\ref{Sec2.1} this approach to get a semigroup in $L_p(J,E_0)$ for $p\in [1,\infty)$ and then show in Section~\ref{Sec2.2} the asynchronous exponential  growth of this semigroup when $p\in (1,\infty)$.\\

\noindent To study \eqref{E1a} we assume that 
\Links
\begin{equation}\tag*{${\bf (A_1)}$}
\begin{split} 
& \text{there is $\rho>0$ such that}\ A\in C^\rho\big(J,\mathcal{L}(E_0,E_1)\big) \ \text{and}\ -A(a)\ \text{is the generator}\\
& \text{of an analytic  positive semigroup on $E_0$ with domain $E_1$ for each $a\in J$}\,.
\end{split}
\end{equation}
We further let the birth rate $\beta$ be such that
\begin{equation}\tag*{${\bf (A_2)}$}
\begin{split} 
\beta\in L_{\infty}(J,\ml(E_\theta))\,, \quad \theta\in [0,1] \,,\qquad \beta(a)\in \ml_+(E_0)\, ,\quad a\in J\,.
\end{split}
\end{equation}
\Rechts
For some of the subsequent results not all of these assumptions (and the ones to come) are required (in this strictness). We refer to the original papers for more precise assumptions and corresponding relaxations.\\

According to \cite[II. Corollary 4.4.2, II. Theorem 6.4.2]{LQPP}, assumption $(A_1)$ implies that $A$ generates a {\it positive parabolic evolution operator} 
$$
\Pi(a,\sigma)\,,\quad 0\le\sigma\le a<a_m\,,
$$ 
on $E_0$ with regularity subspace $E_1$. Hence, given $\sigma\in [0,a_m)$ and $v^0\in E_0$, the unique solution 
$$
v\in C^1((\sigma,a_m),E_0)\cap C((0,a_m),E_1)\cap C([\sigma,a_m),E_0)
$$ to
\bqn\label{10aac}
\partial_a v +A(a) v=0\,,\quad a\in (\sigma,a_m)\,,\qquad v(\sigma)=v_0\,,
\eqn
is $$v(a)=\Pi(a,\sigma)v_0\,,\quad a\in [\sigma, a_m)\,.$$ Recall that the operator $A$ combines the information on spatial dispersion and mortality and thus, analogously to the spatially homogeneous setting, $\Pi$ contains information on the survivability and spatial distribution of individuals.

Owing to \cite[II. Lemma 5.1.3]{LQPP} there are $M\ge 1$ and $\varpi\in\R$ such that
\bqn\label{10aa}
\|\Pi(a,\sigma)\|_{\ml (E_\alpha)}+(a-\sigma)^{\alpha-\gamma_1}\|\Pi(a,\sigma)\|_{\ml (E_\gamma,E_\alpha)}\le Me^{-\varpi (a-\sigma)}\ ,\quad 0\le \sigma\le a< a_m\ , 
\eqn
for $0\le\gamma_1\le\gamma<\alpha\le 1$ with $\gamma_1<\gamma$ if $\gamma>0$.

\subsection{The Age-Diffusion Semigroup}\label{Sec2.1}

To associate with \eqref{E1a} a linear semigroup on $L_p(E_0)$  for \mbox{$p\in [1,\infty)$}, we  integrate \eqref{1a} along characteristics and so obtain  (parabolic) equations of the form \eqref{10aac} that we can solve with the help of the evolution operator $\Pi$. Using \eqref{2a}-\eqref{3a} we formally derive in this way  that the solution
$$
[\mS(t)\phi](a):=u(t,a)\,,\qquad t\ge 0\,,\quad a\in J\,,
$$ 
to \eqref{E1a} is given by
  \bqn\label{100}
     \big[\mS(t) \phi\big](a)\, :=\, \left\{ \begin{aligned}
    &\Pi(a,a-t)\, \phi(a-t)\ ,& & 0\le t\le a<a_m\ ,\\
    & \Pi(a,0)\, B_\phi(t-a)\ ,& & 0\le a<a_m\, ,\, t>a\ ,
    \end{aligned}
   \right.
    \eqn
where $B_\phi:=u(\cdot,0)$ satisfies the Volterra equation 
    \bqn\label{500}
    B_\phi(t)\, =\, \int_0^t h(a)\, \beta(a)\, \Pi(a,0)\, B_\phi(t-a)\ \rd
    a\, +\, \int_0^{a_m-t} h(a)\, \beta(a+t)\, \Pi(a+t,a)\, \phi(a)\ \rd a\ ,\quad
    t\ge 0\ ,
    \eqn
with cut-off function $h(a):=1$ if $a\in(0,a_m)$ and $h(a):=0$ otherwise. Note that
  \bqn\label{6a}
    B_\phi(t)\, =\, \int_0^{a_m}  \beta(a)\, \big[\mS(t)\phi\big](a)\ \rd a\ ,\quad t\ge 0\ .
    \eqn
Equations \eqref{100}-\eqref{500} correspond to equations \eqref{6}-\eqref{8} from the spatially homogeneous case.
It can be shown \cite[Lemma 2.1]{WalkerMOFM} that there exists indeed a mapping 
$$
[\phi\mapsto B_\phi]\in\ml \big(L_p(E_0),
C(\R^+,E_0)\big)
$$ 
such that $B_\phi$ is the unique solution to~\eqref{500}, and if
 $\phi\in L_p^+(E_0)$, then $B_\phi(t)\in E_0^+$ for $t\ge 0$. From this one deduces by direct computations \cite[Theorem 2.2, Lemma 3.2]{WalkerMOFM}, \cite[Theorem 4]{WebbSpringer}:

\begin{theorem}[Linear Semigroup]\label{T1}
Suppose $(A_1), (A_2)$, and $p\in [1,\infty)$. Then
\mbox{$\mS:=\{\mS(t)\,;\,  t\ge 0 \}$} defined in \eqref{100}-\eqref{500} is a strongly continuous, positive, and eventually compact semigroup on $L_p(E_0)$.
\end{theorem}

In fact, $\mS(t)$ is compact for $t>a_m$ as can be seen from \eqref{100} by Kolmogorov's compactness criterion \cite[Lemma 3.2]{WalkerMOFM}. 

One of the main advantages of the explicit formula for $\mS(t)$ is that  regularizing properties (with respect the ``$x$-variable'')  inherited from the parabolic evolution operator stated in \eqref{10aa} are readily obtained. For instance, \eqref{10aa}, \eqref{6a}, and the singular Gronwall inequality \cite[II.Corollary 3.3.2]{LQPP} imply
\begin{equation}\label{smoothing}
\|\mS(t)\phi\|_{L_p(E_\theta)} +\|B_\phi(t)\|_{E_\theta}  \le c(\theta) \, t^{-\theta}\, e^{t(-\varpi+\vartheta(\theta))}\, \|\phi\|_{L_p(E_0)}\, , \quad t>0\, ,\quad\phi\in L_p(E_0)\,,
\end{equation}
for $\theta\in [0,1/p)$, where $\vartheta(\theta):=(1+\theta) M \|b\|_{L_\infty(J,\ml(E_\theta))}$. Such regularizing effects play an important role in the study of nonlinear models, see e.g. \cite{LaurencotWalker_Proteus,WalkerDCDSA10,WalkerDIE,WalkerEJAM}. \\

In the following, we let $-\mathbb{A}$ denote the generator of the semigroup $\mS$ on $L_p(E_0)$ with domain $\mathrm{dom}(-\A)$. Since $\mS$ is eventually compact, the spectrum of $-\A$ is a pure point spectrum \cite[V. Corollary 3.2]{EngelNagel} and the {\it growth bound} 
$$
\omega(-\A):=\lim_{t\rightarrow\infty}\frac{\log\|\mS(t)\|_{\mathcal{L}(E_0)}}{t}=\inf_{t>0}\frac{\log\|\mS(t)\|_{\mathcal{L}(E_0)}}{t}
$$
and the {\it spectral bound}
$$
s(-\A):=\sup\{\mathrm{Re}\,\lambda\,;\, \lambda\in \sigma(-\A)\}
$$
coincide \cite[IV.Corollary 3.12]{EngelNagel}. Moreover, since $\mS$ is positive, this is a spectral value of $-\A$ provided that $E_0$ is a Banach lattice \cite[Corollary 12.9]{BatkaiFijavzRhandi}:

\begin{corollary}\label{C3}
Suppose $(A_1), (A_2)$, and let $p\in [1,\infty)$. Then $\omega(-\A)=s(-\A)$ and, if $E_0$ is a Banach lattice, this is an eigenvalue of $-\A$.
\end{corollary}

In general it does not seem to be possible to determine the domain  $\mathrm{dom}(-\A)$ without further assumptions, and only a core (that is, a subspace which is dense in the Banach space $\mathrm{dom}(-\A)$ equipped with the operator graph norm) can be provided \cite[Proposition 2.2]{WalkerDIE} (see also \cite{Rhandi}). In some applications a more precise characterization is desired. For this purpose we strengthen assumption $(A_1)$ and now require that the operator $A$ possesses the maximal $L_p$-regularity property already alluded to in the introduction. We refer e.g. to \cite{LQPP,Dore,JPQuasi} for more on this. \\

If $p\in (1,\infty)$ is fixed, we put \mbox{$\varsigma:=\varsigma(p):=1-1/p$} and set 
\begin{equation}\label{Esigma}
E_{\varsigma}:=(E_0,E_1)_{\varsigma,p}\,, 
\end{equation}
where $(\cdot,\cdot)_{\varsigma,p}$ is the real interpolation functor \cite{LQPP,Triebel}. Moreover, introducing
\begin{equation*}
\bbWp:= W_p^1(E_0)\cap L_p(E_1)
\end{equation*}
we recall that
\bqn\label{BUC}
\bbWp\hookrightarrow BUC(E_{\varsigma})
\eqn
according to, e.g. \cite[III.Theorem 4.10.2]{LQPP}, where $BUC(E_{\varsigma})$ stands for the (bounded and uniformly) continuous $E_{\varsigma}$-valued functions on $J$. In particular, the trace $\gamma_0 u:= u(0)$ is well-defined for $u\in \bbWp$ and 
\begin{equation}\label{gamma0}
\gamma_0\in\ml\big(\bbWp,E_\varsigma\big)\,.
\end{equation} 
We then introduce the following crucial assumption:
\Links
\begin{equation}\tag*{${\bf (A_3)}$}
\begin{split} 
&\text{$p\in (1,\infty)$ and the operator}\ A\  \text{has {\it maximal}}\ L_p\text{{\it -regularity}, that is,}\\ 
& (\partial_a+A,  \gamma_0)\in \mathrm{Isom}\big(\bbWp, L_p(E_0)\times E_{\varsigma}\big)\,, 
\end{split}
\end{equation}
\Rechts
where $\mathrm{Isom}(X,Y)$ denotes the isomorphisms between the Banach spaces $X$ and $Y$. 
This implies that, for any $(f,v^0)\in L_p(E_0)\times E_{\varsigma}$ and $\lambda\in \C$, the unique solution $v$ to
$$
\partial_a v +(\lambda+A(a)) v=f(a)\,,\quad a\in (0,a_m)\,,\qquad v
(0)=v^0\,,
$$
 given by
\begin{equation}\label{S}
v(a)= \Pi_\lambda(a,0)v^0 +\int_0^a \Pi_\lambda(a,\sigma)\,f(\sigma)\, \rd \sigma\,,\qquad a\in J\,,
\end{equation}
with
$$
\Pi_\lambda(a,\sigma):=e^{-\lambda(a-\sigma)}\Pi(a,\sigma)\,,\quad 0\le \sigma\le a\le a_m\,,
$$
belongs to $\bbWp$. In particular, $\partial_a v\in L_p(E_0)$ and $Av\in L_p(E_0)$.\\

Writing the resolvent of the generator $-\A$ by means of the Laplace transform formula and using \eqref{100} we get for $\mathrm{Re}\,\lambda$ large and $\psi\in L_p(E_0)$ the identity
\begin{equation}\label{res}
\big[(\lambda+\A)^{-1}\psi\big](a)=\int_0^\infty e^{-\lambda t}\, \big[\mS(t)\, \psi\big](a)\, \rd t\,= v_\lambda(a)+w_\lambda(a)
\end{equation}
for $a\in J$ with
$$
v_\lambda(a):= \Pi_\lambda(a,0) \int_0^\infty e^{-\lambda \sigma} B_\psi(\sigma)\, \rd\, \sigma\,,\,,\qquad w_\lambda(a):= \int_0^a \Pi_\lambda(a,\sigma)\,\psi(\sigma)\, \rd \sigma
$$
so that \eqref{S} and assumption $(A_3)$ imply that both $v_\lambda$ and $w_\lambda$ belong to $\bbWp$. Therefore, $\mathrm{dom}(-\A)\subset \bbWp$ and, using \eqref{100} and \eqref{6a}, we obtain a precise characterization of the generator $-\A$, see
\cite[Theorem 2.8]{WalkerMOFM}:

\begin{theorem}[Generator Characterization]\label{T2}
Assume $(A_1)-(A_3)$. Then the generator $-\mathbb{A}$ of the semigroup $\mS$ on $L_p(E_0)$ is given by
\begin{align*}
&\mathrm{dom}(-\A)=\left\{\phi\in \bbWp\,;\,
 \phi(0)=\int_0^{a_m}\beta(a)\,\phi(a)\, \rd a\right\}\,,\\
& \A \phi=\partial_a\phi +A\phi\,,\quad\phi\in
\mathrm{dom}(-\A)\,.
\end{align*}
\end{theorem}

Theorem~\ref{T2} shows that $\mS$ (being derived formally in the first instance) is indeed the correct choice for the solution operator corresponding to \eqref{E1a}. In combination with Theorem~\ref{T1} we deduce that for any initial value $\phi\in \bbWp$ satisfying 
$$
\phi(0)=\int_0^{a_m}\beta(a)\,\phi(a)\, \rd a\,,
$$
the unique strong solution $u$ to \eqref{E1a}  satisfies
$$
u\in C\big(\R^+,\bbWp\big)\cap C^1(\R^+,L_p(E_0))\,,\qquad u(t)=\mS(t)\phi\,,\quad t\ge 0\,.
$$ 
If $\phi$ belongs merely to $L_p(E_0)$, then $u(t)=S(t)\phi$,\, $t\ge 0$, defines a mild solution in $C(\R^+,L_p(E_0))$. Moreover, $u(t)\in L_p^+(E_0)$ for $t\ge 0$ if $\phi\in L_p^+(E_0)$. 

Together with Kato's theory on evolution operators  \cite{Kato1,Pazy}, Theorem~\ref{T2} paves also the way to consider problems involving time-dependent operators $A=A(t,a)$, see \cite[Proposition 2.10]{WalkerMOFM} for more details (and \cite{RhandiSchnaubelt_DCDS99, WalkerDCDSA10} for related results).\\

With regard to qualitative aspects of solutions it is important to have more information on the spectrum of $-\A$. Actually, we can now either argue as in Corollary~\ref{C3} with the eventual compactness of $\mS$ and \cite[V. Corollary 3.2]{EngelNagel}  or else with the fact that $-\A$ has a compact resolvent by \cite[Corollary 4]{Simon} and Theorem~\ref{T2} (due to the compact embedding of $E_1$  in $E_0$ and $a_m<\infty$): both facts imply that the spectrum of $-\A$ is a pure point spectrum, countable with no finite accumulation point, and any eigenvalue of $-\A$ is a pole of the resolvent $(\lambda+\A)^{-1}$.

Moreover, if $\lambda\in\C$ is an eigenvalue of $-\A$, i.e. $(\lambda+\A)\phi=0$ with $\phi\in \mathrm{dom}(-\A)\setminus\{0\}$, then Theorem~\ref{T2} and \eqref{S} readily imply that
\bqn\label{kerA}
\phi(a)=\Pi_\lambda(a,0)\phi(0)\,,\quad a\in J\,, \qquad \phi(0)=Q_\lambda \phi(0)\, ,
\eqn 
where the operator $Q_\lambda\in\ml(E_0)$ is defined as
\bqn\label{Qlambda}
Q_\lambda:=\int_0^{a_m}\beta(a)\,  \Pi_\lambda(a,0)\ \rd a\, .
\eqn 
Clearly, \eqref{kerA} means that $1$ is an eigenvalue of $Q_\lambda$ with eigenvector $\phi(0)$. It is easily seen that also the converse is true and that the geometric multiplicities of the eigenvalues are the same \cite[Lemma 3.1]{WalkerMOFM}. 
Now, assumption $(A_2)$, \eqref{10aa}, and \eqref{Qlambda} warrant the smoothing property
\bqn\label{Qcomp11}
Q_\lambda\in\mathcal{L}(E_0,E_\theta)\cap \mathcal{L}(E_{1-\theta},E_1)\, ,\qquad\theta\in [0,1)
\eqn 
and hence
\bqn\label{Qcomp1}
Q_\lambda\in\mathcal{K}(E_\theta)\,,\quad \theta\in [0,1)\,.
\eqn 
Therefore, \mbox{$\sigma(Q_\lambda\vert_{E_\theta})\setminus\{0\}$} consists only of eigenvalues and is independent of $\theta\in [0,1)$. 
Consequently, we obtain:

\begin{corollary}\label{C1a}
Assume $(A_1)-(A_3)$. Then the set
$$
\sigma(-\A)=\sigma_p(-\A)=\{\lambda\in\C\,;\, 1\in\sigma_p(Q_\lambda)\}
$$
is countable with no finite accumulation point and any $\lambda \in \sigma_p(-\A)$ is a pole of the resolvent of $-\A$.
Moreover, the geometric multiplicities of the eigenvalues $\lambda\in \sigma_p(-\A)$ and $1\in\sigma_p(Q_\lambda)$ coincide.
\end{corollary}

Recall that $\mathrm{ker}(\A)$ consists exactly of the equilibrium (i.e. time-independent) solutions to~\eqref{E1a}. Thus, taking $\lambda=0$ in \eqref{kerA}, we obtain:

\begin{remark}\label{R3}
There is an  equilibrium  solution $u\in\mathrm{dom}(-\A)$ to \eqref{E1a} if and only if $1\in\sigma_p(Q_0)$. In this case we have, analogously to \eqref{EE0}, that $$u(a)=\Pi(a,0)u(0)\,,\quad a\in J\,,$$ with $u(0)\in E_1\cap \mathrm{ker}(1-Q_0)$. If $u(0)$ belongs to $E_0^+$, then $u(a)\in E_0^+$ for $a\in J$.
\end{remark}

The analogue interpretation to \eqref{10} is that $Q_0$ contains information about the spatial distribution of the average number of offspring per individual over the entire lifespan of the individual. In this sense, $Q_0$ is  a ``spatial reproduction operator''.

\subsection{Asynchronous Exponential Growth}\label{Sec2.2}

We shall give the analogue of the renewal theorem \eqref{renewal}. 
More precisely, we show that there is a real value $\lambda_0$ (given by the growth respectively spectral bound of $-\A$) such that the semigroup $\mS$ is exponentially decreasing with growth rate $\lambda_0<0$ or has {\it asynchronous exponential growth with intrinsic growth constant} $\lambda_0>0$, that is, $e^{-\lambda_0 t}\mS(t)$ converges exponentially to some nonzero rank one projection in $\mathcal{L}(E_0)$ as $t\rightarrow\infty$. The sign of $\lambda_0$ will be given by the sign of $r(Q_0)-1$, where we recall that the {\it spectral radius} of an operator $T\in\mathcal{L}(E_0)$ is
$$
r(T):=\lim_{n\rightarrow\infty} \|T^n\|_{\mathcal{L}(E_0)}^{1/n}\,.
$$ 
In particular, we determine the (in-)stability of the trivial equilibrium in terms of the spectral radius $r(Q_0)$.\\

For this purpose we fix
throughout this subsection  $p\in (1,\infty)$ and assume $(A_1)-(A_3)$. 
Furthermore, we shall impose that
\Links
\begin{equation}\tag*{${\bf (A_4)}$}
\begin{split} 
E_0 \ \text{is a Banach lattice}
\end{split}
\end{equation}
and that\footnote{Recall that if $E$ is an ordered Banach space, then $T\in \ml(E)$ is {\it strongly positive} if $Tz\in E$ is a {\it quasi-interior point} for each $z\in E^+\setminus\{0\}$, that is, if $\langle z',Tz\rangle_{E} >0$ for every $z'\in (E')^+\setminus\{0\}$.}
\begin{equation}\tag*{${\bf (A_5)}$}
\begin{split} 
\beta(a)\Pi(a,0)\in \ml_+(E_0) \ \text{is strongly positive for $a$ in a subset of $J$ of positive measure}\ .
\end{split}
\end{equation}
\Rechts

\noindent Assumption $(A_5)$ and \eqref{Qlambda}, \eqref{Qcomp1} entail that, for $\lambda\in\R$, 
$$
Q_\lambda\in\mathcal{K}(E_0)\ \text{is strongly positive} \,.
$$
This property and the celebrated Kre\u{\i}n-Rutman Theorem (e.g. see \cite[Theorem 3.2]{AmannSIAM}, \cite[Theorem 12.3, Corollary 12.4]{DanersKochMedina}) along with \eqref{Qcomp11} imply the next result.

\begin{lemma}\label{L0} Suppose $(A_1)-(A_5)$. For $\lambda\in\R$,  the spectral radius $r(Q_\lambda)$ is positive and a simple eigenvalue of $Q_\lambda\in\ml(E_0)$ with an  eigenvector $\zeta_\lambda \in E_1$ that is quasi-interior in $E_0^+$. It is the only eigenvalue of $Q_\lambda$ with a positive eigenvector. Moreover, $r(Q_\lambda)$ is an eigenvalue of the dual operator $Q_\lambda'\in\ml(E_0')$ with a positive eigenfunctional $\zeta_\lambda'\in E_0'$.  Finally, for every $ \psi\in E_0^+\setminus\{0\}$, the equation
$(\xi-Q_\lambda)\phi=\psi$
has exactly one positive solution $\phi$ if $\xi > r(Q_\lambda)$ and no positive solution $\phi$ for $\xi \le r(Q_\lambda)$.
\end{lemma}

Using the previous result we can derive very precise information on the dependence of the spectral radius $r(Q_\lambda)$ on $\lambda$. 

\begin{corollary}\label{L1}
Suppose $(A_1)-(A_5)$.
The mapping  $$\R\rightarrow (0,\infty)\,,\quad \lambda\mapsto r(Q_\lambda)$$ is continuous and strictly decreasing with \mbox{$\lim_{\lambda\rightarrow\infty}r(Q_\lambda) =0$} and \mbox{$\lim_{\lambda\rightarrow-\infty}r(Q_\lambda) =\infty$}.
\end{corollary}

\begin{proof}
 Given $\xi>\lambda$, it readily follows from the monotonicity of the function $[\lambda\mapsto e^{-\lambda a}]$ that $Q_\lambda  \gg Q_\xi$, that is, the operator $Q_\lambda -Q_\xi\in \ml(E_0)$ is strongly positive, hence
$$
\langle \zeta_\lambda',Q_\lambda \zeta_\xi\rangle_{E_0}\,>\,\langle \zeta_\lambda',Q_\xi \zeta_\xi\rangle_{E_0} \,.
$$
Therefore, using Lemma~\ref{L0},
$$
r(Q_\lambda)\,\langle \zeta_\lambda',\zeta_\xi\rangle_{E_0}\, =\,\langle Q_\lambda'\zeta_\lambda', \zeta_\xi\rangle_{E_0}\, =\, \langle \zeta_\lambda',Q_\lambda \zeta_\xi\rangle_{E_0}\,>\,\langle \zeta_\lambda',Q_\xi \zeta_\xi\rangle_{E_0} =\, r(Q_\xi)\, \langle \zeta_\lambda',\zeta_\xi\rangle_{E_0} \ ,
$$
hence $r(Q_\lambda)>r(Q_\xi)$ so that $[\lambda\mapsto r(Q_\lambda)]$ is strictly decreasing. Now, if $\lambda\in \R$ is fixed and $\varepsilon>0$, we similarly have
$$
e^{\varepsilon a_m}\, Q_{\lambda} \gg Q_{\lambda-\varepsilon} \,,\qquad Q_{\lambda+\varepsilon} \gg e^{-\varepsilon a_m}\, Q_{\lambda}
$$
and thus, 
$$
e^{\varepsilon a_m}\, r(Q_{\lambda})  > r(Q_{\lambda-\varepsilon}) > r(Q_{\lambda}) > r(Q_{\lambda+\varepsilon}) > e^{-\varepsilon a_m}\, r(Q_{\lambda})\,.
$$
Letting $\varepsilon\rightarrow 0$ implies  the continuity of \mbox{$\lambda\mapsto r(Q_\lambda)$}. Next, $(A_2)$ and \eqref{10aa} ensure
$$
0< r(Q_\lambda)\le \|Q_\lambda\|_{\mathcal{L}(E_0)}\rightarrow0\ ,\quad \lambda\rightarrow\infty\,.$$ 
Finally, there is $\delta\in (0,a_m)$ such that
$$
M:=\int_\delta^{a_m}\beta(a)\,  \Pi(a,0)\, \rd a\in \mathcal{K}(E_0)\ \text{is strongly positive}\ .
$$
Thus, if $\lambda<0$, then $Q_{\lambda} \ge  e^{-\lambda \delta} M$ from which  $r(Q_{\lambda})\ge e^{-\lambda\delta} r(M)>0$. Consequently, \mbox{$\lim_{\lambda\rightarrow-\infty}r(Q_\lambda) =\infty$}.
\end{proof}

It readily follows from Lemma~\ref{L1} that there is a unique $\lambda_0\in\R$ such that 
\begin{equation}\label{lambda0}
r(Q_{\lambda_0})=1\,,
\end{equation}
which is the analogue to \eqref{chareq}.

\begin{corollary}
Suppose $(A_1)-(A_5)$ and let $\lambda_0\in\R$ satisfy \eqref{lambda0}. Then $\lambda_0=s(-\A)=\omega(-\A)$  is a simple eigenvalue of~$-\A$.
\end{corollary}

\begin{proof}
(i) Recall from Lemma~\ref{L0} that  $r(Q_{\lambda_0})=1$ is a simple eigenvalue of $Q_{\lambda_0}$with positive eigenvector $\zeta_{\lambda_0}$. Thus, by Corollary~\ref{C1a}, $\mathrm{ker}(\lambda_0+\A)$ is one-dimensional and spanned by \mbox{$\varphi:=\Pi_{\lambda_0}(\cdot,0)\zeta_{\lambda_0}$}. In order that $\lambda_0$ is simple, it remains to show that $\mathrm{ker}(\lambda_0+\A)^2\subset \mathrm{ker}(\lambda_0+\A)$. Let \mbox{$\psi\in\mathrm{ker}(\lambda_0+\A)^2$} and set $$\phi:=(\lambda_0+\A)\psi\in \mathrm{ker}(\lambda_0+\A)\ .$$ Then $\phi=\alpha\varphi$ for some $\alpha\in\R$. Suppose $\alpha\not= 0$, so without loss of generality $\alpha>0$. Let $\tau>0$ be such that $\tau\zeta_{\lambda_0}+\psi(0)\in E_\varsigma^+\setminus\{0\}$ and put $v:=\tau\varphi+\psi\in\mathrm{dom}(-\A)$. Then $(\lambda_0+\A)v=\phi$ and from Theorem~\ref{T2} it follows that
$$
v(a)=\Pi_{\lambda_0}(a,0) v(0)+\alpha\int_0^a\Pi_{\lambda_0}(a,\sigma)\,\Pi_{\lambda_0}(\sigma,0)\,\zeta_{\lambda_0}\ \rd \sigma= \Pi_{\lambda_0}(a,0) v(0)+ \alpha\,a\, \Pi_{\lambda_0}(a,0)\,\zeta_{\lambda_0}
$$
for $a\in J$ and
$$
v(0)=\int_0^{a_m} \beta(a)\, v(a)\ \rd a \ .
$$
Plugging the former into the second formula yields
$$
(1-Q_{\lambda_0}) v(0)=\alpha \int_0^{a_m} \beta(a)\, a\, \Pi_{\lambda_0}(a,0)\,\zeta_{\lambda_0}\ \rd a\ .
$$
As $v(0)$ and the right-hand side are both positive and nonzero, we derive from Lemma~\ref{L0} a contradiction to \mbox{$r(Q_{\lambda_0})=1$}. Consequently, $\alpha=0$ and hence $\phi=0$. Therefore, $\lambda_0$ is a simple eigenvalue of $-\A$.

(ii) It follows from (i) and Corollary~\ref{C3} that $\lambda_0\le s:=s(-\A)=\omega(-\A)$. Moreover, Corollary~\ref{C3}, Corollary~\ref{C1a}, and $(A_4)$ imply that $1\in \sigma_p(Q_s)$, hence $r(Q_s)\ge 1 = r(Q_{\lambda_0})$ from which $\lambda_0\ge s$. This yields the assertion.
\end{proof}

In summary, we now know that $\mS$ is a strongly continuous eventually compact semigroup on $L_p(E_0)$ and that $\lambda_0$ is a dominant eigenvalue of the generator $-\A$ and a first-order pole of the resolvent. Let 
\bqn\label{PPP}
P_{\lambda_0}:=\frac{1}{2\pi i}\int_{\vert\lambda-\lambda_0\vert=\epsilon} \, (\lambda+\A)^{-1}\,\rd \lambda
\eqn
for $0<\epsilon<\mathrm{dist}\big(\lambda_0,\sigma(-\A)\setminus\{\lambda_0\}\big)$ denote the corresponding rank one spectral projection  $$P_{\lambda_0}: L_p(E_0)\rightarrow \mathrm{ker}(\lambda_0 +\A)$$ with $\mathrm{ker}(\lambda_0 +\A)=\mathrm{span}\{\Pi_{\lambda_0}(\cdot,0)\zeta_{\lambda_0}\}$. The Residue Theorem and the fact that $\lambda_0$ is a simple pole of the resolvent allow one to derive from \eqref{res}  actually  an (almost) explicit formula for the spectral projection $P_{\lambda_0}$ (see \cite[Proposition 3.8]{WalkerMOFM}), that is,
\bqn\label{PP}
P_{\lambda_0}\phi=\frac{\langle \zeta_{\lambda_0}',H_{\lambda_0}\phi\rangle}{\langle \zeta_{\lambda_0}',\int_0^{a_m} a \beta(a)\Pi_{\lambda_0}(a,0)\rd a\,\zeta_{\lambda_0}\rangle}\Pi_{\lambda_0}(\cdot,0)\zeta_{\lambda_0}\ 
\eqn
for $\phi\in L_p(E_0)$, where 
$$
 H_{\lambda_0}\phi:=\int_0^{a_m}\beta (s)\, \int_0^s\Pi_{\lambda_0}(s,\sigma)\,\phi(\sigma)\ \rd\sigma\,\rd s
 $$
and $\zeta_{\lambda_0} \in E_1^+$ and $\zeta_{\lambda_0}'\in \mathcal{L}(E_0')$  stem from Lemma~\ref{L0}.

We are now in a position to apply \cite[Proposition 2.3]{Webb_TAMS_87} (see also \cite[V. Corollary 3.3]{EngelNagel}) and obtain the asynchronous exponential growth of the semigroup $\mS$:

\begin{theorem}[Asynchronous Exponential Growth]\label{C14}
Suppose $(A_1)-(A_5)$ and let $\lambda_0\in\R$ satisfy \eqref{lambda0}. Then there are $\delta>0$ and $N\ge 1$ such that
$$
\|e^{-\lambda_0 t}\, \mS(t)- P_{\lambda_0}\|_{\ml(L_p(E_0))}\le N e^{-\delta t}\,,\quad t\ge 0\,,
$$
where $P_{\lambda_0}$ is the projection given by \eqref{PPP} - \eqref{PP}.
\end{theorem}

\noindent Let us emphasize that, by construction of $\lambda_0$ in \eqref{lambda0}, we have (compare with \eqref{xx}) 
$$
\mathrm{sign}(\lambda_0)= \mathrm{sign}(r(Q_{0})-1)\,.
$$
Recall from Corollary~\ref{L1} that $\lambda_0=\omega(-\A)$ and that, if $\lambda_0=0$, then $P_{0}\phi$ belongs to the one-dimensional space $\mathrm{ker}(\A)$ consisting of equilibrium solutions to  \eqref{E1a}. Consequently, we obtain a characterization of the large-time behavior of solutions in dependence on $r(Q_0)$:

\begin{corollary}\label{stable}
Suppose $(A_1)-(A_5)$.  
\begin{itemize}
\item[{\bf (i)}]  If $r(Q_{0})<1$, then the zero equilibrium to \eqref{E1a} is globally exponentially asymptotically stable. 
\item[{\bf (ii)}] If $r(Q_{0})=1$, then  the zero equilibrium to \eqref{E1a} is stable. Moreover, the solution $u$ to  \eqref{E1a} with $\phi\in L_p(E_0)$ converges exponentially toward an equilibrium.
\item[{\bf (iii)}] If $r(Q_{0})>1$, then  the zero equilibrium to \eqref{E1a} is unstable. More precisely, the solution $u$ to  \eqref{E1a} with $\phi\in L_p(E_0)$ is asymptotic to the stable age distribution $e^{\lambda_0 t} P_{\lambda_0}\phi $ with $\lambda_0>0$ satisfying \eqref{lambda0} and $P_{\lambda_0}\phi$ being given by \eqref{PP}.
\end{itemize}
\end{corollary}

Related results to Theorem~\ref{C14} with different approaches, not relying on maximal regularity of the diffusion term,  have been obtained previously elsewhere.  Based on  positive perturbations of semigroups in $L_1(E_0)$, a similar result has been shown in \cite{ThiemeDCDS} which has been recovered as a par\-ti\-cu\-lar case in  \cite{RhandiSchnaubelt_DCDS99} (also see \cite{Rhandi}), where time-dependent operators have been treated by means of perturbation techniques of Miyadera type. Note also that the general results of \cite{ThiemeDCDS} apply as well to other situations than the operator $A$ describing spatial diffusion. We also refer to \cite{AnitaAnita,MagalThieme} for the study of the large-time behavior of nonlinear models (see also below).

\begin{remark}\label{R0}
Some of the previous arguments in this section rely on the fact that the semigroup~$\mS$ is eventually compact which is no longer true if $a_m=\infty$. However, if $a_m=\infty$ it can be shown \cite[Lemma 3.2]{WalkerMOFM} (imposing slightly stricter assumptions, e.g. $\varpi>0$ in \eqref{10aa}) that the semigroup $\mS$ is \emph{quasi-compact} (see \cite{EngelNagel2} for a definition) and that $\lambda_0$ is still a dominant eigenvalue and a  simple pole of the resolvent of $-\A$ so that  \cite[Corollary 4.8]{EngelNagel2} implies the asynchronous exponential  growth of $\mS$. Thus,
Theorem~\ref{C14} remains true in the case $a_m=\infty$. For details we refer to \cite{WalkerMOFM}.
\end{remark}

It is worth noting that Theorem~\ref{C14} is the basis to investigate the asynchronous exponential growth for semilinear equations. More precisely, consider
\begin{subequations}\label{E1ab}
\begin{align}
&\partial_t u\,+\, \partial_au \, +\,     A(a)\,u\, =-\mu(u,a) u\ , && t>0\, ,\quad a\in (0,a_m)\ ,\label{1ab}\\ 
&u(t,0)\, =\, \int_0^{a_m}\beta(a)\, u(t,a)\, \rd a\ ,& & t>0\, ,  \label{2ab}\\
&u(0,a)\, =\,  \phi(a)\ , & &\hphantom{t>0\ ,\quad} a\in (0,a_m)\,,\label{3ab}
\end{align}
\end{subequations}
with a semilinear term on the right-hand side of \eqref{1ab} (representing a nonlinear death process). Suppose that the function $\mu=\mu(a,u)$ satisfies (we indicate for preciseness the interval $J$)
\Links
\begin{equation}\tag*{${\bf (B_1)}$}
\begin{split} 
&\mu: L_p(J,E_0)\rightarrow L_\infty\big(J,\mathcal{L}(E_0)\big)\,,\ u\mapsto \mu(u,\cdot)\ \text{is uniformly Lipschitz continuous} \\
&\text{on bounded sets with}\ \|\mu(u,\cdot)\|_{L_\infty(J,\mathcal{L}(E_0))}\le f(\|u\|_{L_p(J,E_0)})\,,\ u\in L_p(J,E_0)\,,
\end{split}
\end{equation}
\Rechts
where the function $f$ satisfies
\Links
\begin{equation}\tag*{${\bf (B_2)}$}
\begin{split} 
&f :\R^+\rightarrow\R^+ \ \text{ is nonincreasing and }\ \displaystyle\int^\infty \frac{f(r)}{r}\,\rd r <\infty\,.
\end{split}
\end{equation}
\Rechts
Then, given $\phi\in L_p(E_0)$, there is a unique  mild solution $u\in C(\R^+,L_p(E_0))$ to \eqref{E1ab}. We introduce the nonlinear semigroup $\mathbb{T}$ by setting $\mathbb{T}(t)\phi:=u(t)$ and put
$$
Q_{\lambda_0}(\phi):= P_{\lambda_0}\left( \phi+\int_0^\infty e^{-\lambda_0 s} F\big(\mathbb{T}(s)\phi\big)\,\rd s \right)\,,
$$
where $F(v):=-\mu(v,\cdot)v$.
The next result now immediately follows from \cite[Theorem 1.1, Theorem 1.3]{GyllenbergWebb}:

\begin{corollary}
Suppose $(A_1)-(A_5)$ and $(B_1)-(B_2)$. Let $r(Q_0)>1$, that is,  $\lambda_0>0$ in \eqref{lambda0}.  There are $\delta>0$ and $N\ge 1$ such that, if $\mathbb{T}(\cdot)\phi\in C(\R^+,L_p(E_0))$ denotes the mild solution to \eqref{E1ab} for a given $\phi\in L_p(E_0)$, then
$$
\|e^{-\lambda_0 t}\, \mathbb{T}(t)\phi- Q_{\lambda_0}(\phi)\|_{\ml(L_p(E_0))}\le N e^{-\delta t}\|\phi\|_{L_p(E_0)}\,,\quad t\ge 0\,.
$$
\end{corollary}

Thus, asynchronous exponential growth also occurs in the nonlinear model \eqref{E1ab}.

\section{Equilibria for Nonlinear Models}\label{Sec3}

In Remark~\ref{R3} and Corollary~\ref{stable} we have seen that the linear problem \eqref{E1a} admits a non-trivial positive equilibrium solution if and only if the condition $r(Q_0)=1$ is met for the operator $Q_0$ defined in \eqref{Qlambda}.
This condition is very restrictive since it is satisfied only in special cases. One would, however, intuitively expect that equilibria should exist in many populations. One is thus led to investigate nonlinear models that feature nontrivial positive equilibrium solutions under natural assumptions. 
In many situations the birth rate $\beta$ (and death rate $\mu$) but also the spatial migration depend on the (e.g. total) population itself, and the corresponding equilibrium equations become
\begin{subequations}\label{E1ae}
\begin{align}
&\partial_au \, +\,     A(u,a)\,u\, =0\, , \qquad a\in (0,a_m)\, ,\label{1ae}\\ 
&u(0)\, =\, \int_0^{a_m}\beta(u,a)\, u(a)\, \rd a\, ,  \label{2ae}
\end{align}
\end{subequations}
where we recall that the death modulus is included in the operator $A$. We emphasize that \eqref{E1ae} is a quasilinear (parabolic) equation subject to a nonlocal and nonlinear initial condition and that it is posed on an {\it a priori} given interval of existence.

One of the main difficulties in this context lies in the fact that $u\equiv 0$ is always a solution to \eqref{E1ae}, hence any method for constructing equilibria has to rule out that one hits this trivial equilibrium.
We are presenting herein two possible approaches to construct nontrivial positive equilibrium solutions: a fixed point argument \cite{WalkerJDE10}  in Section~\ref{Sec3.1} and a bifurcation approach \cite{DelgadoEtAl2,WalkerSIMA09,WalkerJDE10,WalkerJDDE} in Section~\ref{Sec3.2}. Of course, there are also other possible approaches, e.g. see \cite{DS3} for an application of the sub-supersolution method for a model with linear diffusion. We  refrain from considering the well-posedness of the nonlinear version of \eqref{E1a} (that is, the time-dependent version of \eqref{E1ae}), but refer to \cite{WalkerDCDSA10} for a rather thorough treatment of this issue.

\subsection{Equilibria: Application of a Fixed Point Argument in Conical Shells}\label{Sec3.1}

A promising technique to prove the existence of nontrivial positive solutions to \eqref{E1ae} is the use of a fixed point argument in conical shells \cite{AmannSIAM}. Such an approach was first applied in the spatially homogeneous case in \cite[Theorem 1]{Pruess83} (see also \cite[Theorem 4.1]{WebbBook}). Later it was carried over in \cite[Theorem 3.1]{WalkerJDE10} to the spatially inhomogeneous ``quasilinear'' case \eqref{E1ae}. Differently from \cite{WalkerJDE10}, where the result was proven in the phase space $L_1(E_0)$, we present here the result in a functional setting based on maximal regularity, that is, in $L_p(E_0)$ with $p\in (1,\infty)$, which  will also be used in the subsequent section.\\

Let $p\in (1,\infty)$ be fixed. We recall the definition of $E_\varsigma$ in \eqref{Esigma} and suppose that its positive cone has nonempty interior\footnote{Recall that $\mathrm{int}(E_\varsigma^+)\not=\emptyset$ consists exactly of the quasi-interior points if $E_\varsigma$.}, i.e.
\Links
\begin{equation}\tag*{${\bf (A_6)}$}
\begin{split} 
\mathrm{int}(E_\varsigma^+)\not=\emptyset\,.
\end{split}
\end{equation}
\Rechts
We also assume that there is a Banach space $X$ such that
\Links
\begin{equation}\tag*{${\bf (A_7)}$}
\begin{split} 
\bbWp \dhr X\hookrightarrow L_p(E_0)
\end{split}
\end{equation}
\Rechts
and  
\Links
\begin{equation}\tag*{${\bf (A_8)}$}
\begin{split} 
 A: X\rightarrow \mathcal{L}\big( \bbWp, L_p(E_0)\big)\,,\ u\mapsto A(u) \quad \text{is continuous}\,,
\end{split}
\end{equation}
\Rechts
where  $A(u):=A(u,\cdot)$ is supposed to have maximal $L_p$-regularity, that is,
\Links
\begin{equation}\tag*{${\bf (A_9)}$}
\begin{split} 
(\partial_a+A(u),  \gamma_0)\in \mathrm{Isom}\big(\bbWp, L_p(E_0)\times E_{\varsigma}\big)\,, \quad u\in X\,. 
\end{split}
\end{equation}
\Rechts
This implies that the map
$$
T: X \rightarrow \ml\big (L_p(E_0)\times E_{\varsigma},\bbWp\big)\,,\quad 
u\mapsto T[u]:=\big(\partial_a+A(u),\gamma_0\big)^{-1}
$$
is continuous due to continuity of the inversion map~\mbox{$B\mapsto B^{-1}$} for linear operators. In particular, setting $$\Pi[u]:=T[u](0,\cdot)\in \ml\big(E_\varsigma,\bbWp\big)\,,\quad u\in X\,,$$
we deduce that
\bqn\label{Pi}
\Pi:X\rightarrow \ml\big(E_\varsigma,\bbWp\big)\ \text{is continuous}
\eqn
and, for any $u\in X$ and $\varphi\in E_{\varsigma}$,
$$
v:=\Pi[u]\varphi\in \bbWp
$$ 
is the unique solution  to the linear Cauchy problem
$$
\partial_a v +A(u,a) v=0\,,\quad a\in J\,,\qquad v
(0)=\varphi\,.
$$
It is convenient to write
$\Pi[u](a)\varphi:=(\Pi[u]\varphi)(a)$ for $a\in J$, $\varphi\in E_\varsigma$, and $u\in X$.
Hence $\Pi[u](a)$ is the nonlinear analogue to $\Pi(a,0)$ from the previous section. In the following we set 
$$
X^+:=X\cap L_p^+(E_0)\quad \text{and}\quad \bbWpp:=\bbWp\cap L_p^+(E_0)\,.
$$
We then suppose that the above Cauchy problem admits positive solutions for positive initial values in the sense that
\Links
\begin{equation}\tag*{${\bf (A_{10})}$}
\begin{split} 
\Pi[u]\in \ml_+\big(E_\varsigma,\bbWp\big)\ ,\quad u\in X^+\ ,
\end{split}
\end{equation}
\Rechts
that is, $\Pi[u]$ maps the positive cone $E_\varsigma^+$ into the positive cone $\bbWpp$. As for the birth modulus we assume that (with $1/p+1/p'=1$)
\Links
\begin{equation}\tag*{${\bf (A_{11})}$}
\begin{split} 
\beta: X\rightarrow L_{p'}(\ml_+(E_1,E_\vartheta))\,,\ u\mapsto \beta(u) \quad \text{is continuous for some $\vartheta\in (\varsigma,1]$ .}
\end{split}
\end{equation}
\Rechts
We further assume that:
\Links
\begin{equation}\tag*{${\bf (A_{12})}$}
\begin{split} 
&\text{for each $u\in X^+$, the operator }\ \beta(u,a)\Pi[u](a)\in\ml_+(E_\varsigma) \ \text{is}\\
&\text{strongly positive for $a$ in a subset of $J$ of positive measure}\, .
\end{split}
\end{equation}
\Rechts
It is worth noting that, though the assumptions above are technical, they are not restrictive and satisfied in many applications (see Section~\ref{Sec4}). We also point out that all the results presented in this 
Section~\ref{Sec3} include the case $a_m=\infty$ (without additional assumptions). 

Now, analogously to the previous section we introduce
\begin{equation}\label{Qu}
Q[u]:=\int_0^{a_m} \beta(u,a) \Pi[u](a)\,\rd a \,,\quad u\in X\,,
\end{equation}
and obtain from \eqref{Pi} and $(A_{11})$ that
\begin{equation}\label{Qucont}
Q:  X\rightarrow  \mathcal{L}(E_\varsigma, E_\vartheta)\ \text{is continuous}
\end{equation} 
while $(A_{12})$ and \eqref{dhr}  imply that
\begin{equation}\label{Qucomp}
Q[u]\in \mathcal{K}(E_\varsigma)\ \text{is strongly positive for each $u\in X^+$}\,.
\end{equation}
Hence, the Kre\u{\i}n-Rutman Theorem \cite[Theorem 3.2]{AmannSIAM} yields the analogue of Lemma~\ref{L0}.

\begin{lemma}\label{L0a} Suppose $(A_6)-(A_{12})$. For $u\in X^+$,  the spectral radius $r(Q[u])$ is positive and a simple eigenvalue of $Q[u]\in \mathcal{L}(E_\varsigma)$ with an  eigenvector $\zeta_u \in \mathrm{int}(E_\varsigma^+)$. It is the only eigenvalue of $Q[u]$ with a positive eigenvector. Moreover, for every $ \psi\in E_\varsigma^+\setminus\{0\}$, the equation
$(\xi-Q[u])\phi=\psi$
has exactly one positive solution $\phi$ if $\xi > r(Q[u])$ and no positive solution $\phi$ for $\xi \le r(Q[u])$.
\end{lemma}

\noindent With the notation introduced above we deduce that any solution $u$ to \eqref{E1ae} in $\bbWp$ with $\varphi=u(0)$ (in $E_\vartheta$) satisfies
\begin{equation}\label{equi}
u(a)=\Pi[u](a) \varphi \,,\quad a\in J\,,\qquad \varphi =Q[u]\varphi\,,
\end{equation}
or, equivalently, is a fixed point of the map
$$
\Gamma: \bbWp \times E_\vartheta \rightarrow \bbWp\times E_\vartheta\,,\quad (u,\varphi) \mapsto \big(\Pi[u]\varphi, Q[u]\varphi\big)\,.
$$
Invoking \eqref{Qucont}, \eqref{Pi}, $(A_7)$, and \eqref{dhr} it readily follows that $\Gamma$ is a compact mapping. To apply the fixed point theorem in conical shells we impose the following crucial conditions:
\Links
\begin{equation}\tag*{${\bf (A_{13})}$}
\begin{split} 
&\text{there is $\tau_0>0$ such that $r(Q[u])\ge 1$ for}\\
&\text{$u\in \bbWpp$ with $0<\|u\|_{\bbWp}<\tau_0$}
\end{split}
\end{equation}
\Rechts
and
\Links
\begin{equation}\tag*{${\bf (A_{14})}$}
\begin{split} 
&\text{there is $\tau_1>0$ with $\tau_0\not=\tau_1$ such that $r(Q[u])\le 1$}\\
&\text{for $u\in \bbWpp$ with $\|u\|_{\bbWp}>\tau_1$.}
\end{split}
\end{equation}
\Rechts
Roughly speaking, assumption $(A_{13})$ can be interpreted as that the population has to increase when too small while it has to decrease when too large according to assumption $(A_{14})$.  
Under these assumptions we can prove that there is at least one nontrivial positive solution to the equilibrium equation \eqref{E1ae}:

\begin{theorem}[Equilibria in Conical Shells]
Assume $(A_6)-(A_{14})$. Then there is at least one nontrivial positive solution $u\in \bbWpp$ to \eqref{E1ae}.
\end{theorem}

\begin{proof}
We aim at applying the fixed point theorem from \cite[Theorem 12.3]{AmannSIAM}. 
For the sake of brevity, let $\mathbb{W}_p:=\bbWp$ and let $\bar B_R$ denote the closed ball in $\mathbb{W}_p^+\times E_\vartheta^+$  centered at the origin with radius
$$
R:=\tau_1\left(1+ \sup\left\{\|\beta(v)\|_{L_{p'}(\mathcal{L}(E_1,E_\vartheta))}\,;\, v\in \mathbb{W}_p^+\,,\, \|v\|_{\mathbb{W}_p}\le \tau_1\right\}\right) +\tau_0 <\infty\,.
$$
Then $\Gamma:\bar B_R\rightarrow \mathbb{W}_p^+\times E_\vartheta^+$ is compact. To check condition (i) from \cite[Theorem 12.3]{AmannSIAM} suppose for contradiction that there is $\lambda>1$ and $(u,\varphi)\in \bar B_R$ with
$$
\|u\|_{\mathbb{W}_p}+\|\varphi\|_{E_\vartheta} =R\quad\text{and}\quad \Gamma(u,\varphi)=\lambda (u,\varphi)\,.
$$
Then, by definition of $\Gamma$,
\bqn\label{r1}
\lambda u=\Pi[u]\varphi\,,\quad \lambda \varphi =Q[u]\varphi\,,
\eqn
and, in particular, $\varphi\in E_\vartheta^+\setminus\{0\}$ is an eigenvector of the operator $Q[u]$ to the eigenvalue $\lambda>1$. Thus $r(Q[u])>1$ so that $(A_{14})$ implies that 
\bqn\label{r2}
\|u\|_{L_p(E_1)}\le \|u\|_{\mathbb{W}_p} \le \tau_1\,.
\eqn
But then we deduce from \eqref{r1} and \eqref{r2} that
\begin{equation*}
\begin{split}
 R&=\|u\|_{\mathbb{W}_p}+\|\varphi\|_{E_\vartheta} = \|u\|_{\mathbb{W}_p}+\|\lambda^{-1} Q[u]\varphi\|_{E_\vartheta} \\
& = \|u\|_{\mathbb{W}_p}+\left\|\frac{1}{\lambda}\int_0^{a_m} \beta(u,a) \Pi[u](a) \varphi\,\rd a\right\|_{E_\vartheta} \\
& \le \tau_1 + \int_0^{a_m} \|\beta (u,a)\|_{\mathcal{L}(E_1,E_\vartheta)} \|u(a)\|_{E_1}\,\rd a\\
& \le \tau_1+ \|\beta(u)\|_{L_{p'}(\mathcal{L}(E_1,E_\vartheta))} \| u\|_{L_p(E_1)}\\
&\le \tau_1\left(1+ \sup\left\{\|\beta(v)\|_{L_{p'}(\mathcal{L}(E_1,E_\vartheta))}\,;\, v\in \mathbb{W}_p^+\,,\, \|v\|_{\mathbb{W}_p}\le \tau_1\right\}\right)
\end{split}
\end{equation*}
contracting the fact that $R$ was chosen strictly larger than the right-hand side. To check condition (ii) from \cite[Theorem 12.3]{AmannSIAM} fix any $\psi\in\mathrm{int}(E_\varsigma^+)$ and suppose for contradiction that there is $\lambda>0$ and $(u,\varphi)\in \bar B_R$ with
$$
\|u\|_{\mathbb{W}_p}+\|\varphi\|_{E_\vartheta} =\tau_0\quad\text{and}\quad (u,\varphi)-\Gamma(u,\varphi)=\lambda (0,\psi)\,.
$$
Then, again by definition of $\Gamma$,
$$
u=\Pi[u]\varphi \ \text{ with }\ \|u\|_{\mathbb{W}_p}<\tau_0 \ \text{ and }\ (1-Q[u])\varphi=\lambda\psi\,.
$$
The second equation implies that $r(Q[u])<1$ according to Lemma~\ref{L0a} since both $\varphi$ and $\psi$ are nontrivial elements of $E_\varsigma^+$. But this contradicts assumption $(A_{13})$. Consequently, \cite[Theorem 12.3]{AmannSIAM} implies that $\Gamma$ has at least one fixed point $(u,\varphi)\in \bar B_R$ with
$$
\min\{\tau_0,\tau_1\} < \|u\|_{\mathbb{W}_p}+\|\varphi\|_{E_\vartheta} < \max\{\tau_0,\tau_1\}\,.
$$
This proves the assertion.
\end{proof}

\subsection{Equilibria: A Bifurcation Approach}\label{Sec3.2}

The fact that $u\equiv 0$ is always a solution to \eqref{E1ae} has the advantage that the problem of finding nontrivial positive solutions can also be interpreted as a bifurcation problem, possibly giving more insight into the structure of the equilibria set. 
For this we write the birth modulus in the form $$\eta \beta(u,a)$$ and introduce in this way a bifurcation parameter $\eta>0$ which determines the intensity of the individual's fertility\footnote{Similarly, one can introduce a bifurcation parameter in the death modulus to construct nontrivial equilibria~\cite{DelgadoEtAl2, WalkerAMPA11}.} while the qualitative structure of the fertility  is modeled by the function $\beta$. This approach was introduced for purely age-structured models in \cite{CushingI} (see also \cite{CushingII,CushingIII}) and then for the spatially inhomogeneous quasilinear setting  \eqref{E1ae} in \cite{WalkerSIMA09, WalkerJDE10, WalkerJDDE}. The equilibrium problem now reads 
\begin{subequations}\label{E1aee}
\begin{align}
&\partial_au \, +\,     A(u,a)\,u\, =0\, , \qquad a\in (0,a_m)\, ,\label{1aee}\\ 
&u(0)\, =\, \eta \int_0^{a_m}\beta(u,a)\, u(a)\, \rd a\, ,  \label{2aee}
\end{align}
\end{subequations}
with parameter $\eta>0$. Independent of its value, $u\equiv 0$ is still a solution. \\

We shall impose the same assumptions $(A_6)-(A_{12})$ as before, but strengthen $(A_8)$ and $(A_{11})$ slightly by assuming differentiability of the maps:
\Links
\begin{equation}\tag*{${\bf (A_8')}$}
\begin{split} 
 A: X\rightarrow \mathcal{L}\big( \bbWp, L_p(E_0)\big)\,,\ u\mapsto A(u) \quad \text{is continuously differentiable}\,,
\end{split}
\end{equation}
\Rechts
and
\Links
\begin{equation}\tag*{${\bf (A_{11}')}$}
\begin{split} 
& \beta: X\rightarrow L_{p'}(\ml_+(E_1,E_\vartheta))\,,\ u\mapsto \beta(u) \  \text{ is continuously}\\
&\text{differentiable for some $\vartheta\in (\varsigma,1]$ .}
\end{split}
\end{equation}
\Rechts
As in \eqref{equi}, any solution
$u$ in $\bbWp$ to \eqref{E1aee} with $\varphi=u(0)$ in $E_\varsigma$ satisfies
\begin{equation}\label{equil}
u(a)=\Pi[u](a) \varphi \,,\quad a\in J\,,\qquad \varphi =\eta Q[u]\varphi\,,
\end{equation}
where $\Pi[u]$ and $Q[u]$ are defined in \eqref{Pi} respectively \eqref{Qu}. In particular,
$$
Q[u]=\ell[u]\Pi[u]\,,\quad u\in X\,,
$$
where
\begin{equation}\label{ell}
\ell: X\rightarrow \ml(\bbWp,E_\vartheta) \ \text{is continuously differentiable}
\end{equation}
and given as
$$
\ell[u] v:= \int_0^{a_m}\beta(u,a)\, v(a)\, \rd a\, ,\qquad  v\in \bbWp\,,\quad u\in X\,.
$$
We can actually weaken the strong positivity assumption $(A_{12})$ in that we have to assume it only at $u=0$ (instead of all $u\in X^+$):
\Links
\begin{equation}\tag*{${\bf (A_{12}')}$}
\begin{split} 
&\text{$\beta(0,a)\Pi[0](a) \in \ml_+(E_\varsigma)$ is strongly positive for $a$}\\
&\text{in a subset of $J$ of positive measure}\, .
\end{split}
\end{equation}
\Rechts
Then still $Q[u]\in \ml_+(E_\varsigma)\cap \mathcal{K} (E_\varsigma)$ for $u\in X^+$ by $(A_{10})$, $(A_{11}')$ but only $Q[0]\in \mathcal{K} (E_\varsigma)$ is strongly positive (i.e. \eqref{Qucomp} only holds for $u=0$) and therefore, Lemma~\ref{L0a} is valid only for $u=0$.\\

It is also worth noting that with $$S:=T[0]\in \ml\big (L_p(E_0)\times E_{\varsigma},\bbWp\big)$$ we can write a solution to \eqref{E1aee} in the form
$$
u=S\big(A(0)u-A(u)u \,,\, \eta\ell[u] u\big)\, .
$$
Consequently, introducing the function $F:\R\times \bbWp\rightarrow \bbWp$  as
$$
F(\eta,u):=u-S\big(A(0)u-A(u)u \,,\, \eta\ell[u] u\big)\ ,\quad (\eta,u)\in \R\times \bbWp\,,
$$
the solutions to \eqref{E1aee} we are interested in coincide with the zero set
$$
\mathfrak{S}:=\{(\eta,u)\in \R^+\times \bbWpp\,;\, F(\eta,u)=0\}\, .
$$ 
Clearly, $(\eta,u)=(\eta,0)$ for $\eta\in\R^+$ gives a trivial branch in $\mathfrak{S}$ of solutions to \eqref{E1aee}. We shall next show that  a nontrivial local branch of positive solutions bifurcates from this branch at the point 
$(\eta,u)=(r(Q[0])^{-1},0)$. 
Subsequently, we prove that this local branch is contained in an unbounded continuum of positive solutions.

\subsubsection{Local Bifurcation} We shall apply the celebrated result of Crandall-Rabinowitz \cite[Theorem 1.7]{CrandallRabinowitz} on local bifurcation from simple eigenvalues. We thus need to investigate first the Fr\'echet derivative $F_u(\eta,0)$, given by
$$
F_u(\eta,0)\phi=\phi- \eta\, S(0,\ell[0]\phi)=\phi- \eta\, \Pi[0]\ell[0]\phi
$$
for $\phi\in \bbWp$ and $\eta\in\R$. It readily follows from \eqref{ell}, \eqref{Pi}, and \eqref{dhr} that 
$$K:=  \Pi[0]\ell[0]\in\mathcal{K} (\bbWp)$$  and therefore, that
\begin{equation}\label{FF}
F_u(\eta,0)=1- \eta K
\end{equation}
is a Fredholm operator of zero index by the Riesz-Schauder Theorem:

\begin{proposition}\label{fredholm}
Suppose $(A_6)-(A_{10})$ along with $(A_8')$, $(A_{11}')$,  and $(A_{12}')$. Then, given $\eta\in\R$, 
$$F_u(\eta,0)=1-\eta K \in \ml(\bbWp)$$  is a Fredholm operator of index zero. Its kernel is 
$$
\mathrm{ker}\big(F_u(\eta,0)\big)=\big\{\Pi[0]w\,;\, w\in \mathrm{ker}\big(1-\eta Q[0]\big)\big\}
$$
and its image is
\bqn\label{rg}
\mathrm{rg}\big(F_u(\eta,0)\big)=\big\{h\in\bbWp\,;\, h(0)+\eta\ell [0] S\big(\partial_ah+A(0)h, 0\big)\in \mathrm{rg}\big(1-\eta Q[0]\big)\big\}\,.
\eqn
\end{proposition}

\begin{proof}
 To compute kernel and image of $F_u(\eta,0)$ note that
the equation $F_u(\eta,0)\phi=h$ for $\phi, h\in\bbWp$ is equivalent to
\begin{align}
\partial_a\phi+A(0)\phi&=\partial_a h+A(0)h\ ,\label{L1a}\\
\phi(0)-\eta\ell[0]\phi&=h(0)\ .\label{L2a}
\end{align}
From \eqref{L1a} it follows  that
\bqn\label{u1}
\phi=S\big(\partial_ah+A(0)h, 0\big)+\Pi[0]\phi(0)
\eqn
and, when plugged into \eqref{L2a}, we obtain 
\bqn\label{u2}
(1-\eta Q[0])\phi(0)= h(0)+\eta\ell[0]S\big(\partial_ah+A(0)h, 0\big)\, .
\eqn
Equations \eqref{u1} and \eqref{u2} imply the assertion.
\end{proof}

Proposition~\ref{fredholm} along with Lemma~\ref{L0a} yield 
\begin{equation}\label{bifvalue}
\eta_0:=\frac{1}{r(Q[0])} >0
\end{equation}
as a candidate for a bifurcation value as the kernel of the linearization $F_u(\eta_0,0)$ is one-dimensional. To establish that $\eta_0$ is indeed a bifurcation value we need a further result. Recall from Lemma~\ref{L0a} that $\zeta_0 \in \mathrm{int}(E_\varsigma^+)$ is an eigenvector of $Q[0]\in \mathcal{L}(E_\varsigma)$ corresponding to the simple eigenvalue $r(Q[0])$.

\begin{corollary}
Suppose  $(A_6)-(A_{10})$ along with $(A_8')$, $(A_{11}')$,  and $(A_{12}')$. Let $\eta_0>0$ be given by \eqref{bifvalue}. Then 1 is a simple eigenvalue of the  operator $\eta_0 K\in \mathcal{K}(\bbWp)$ and
\begin{equation}\label{kern}
\mathrm{ker}(F_u(\eta_0,0))=\mathrm{ker}(1-\eta_0 K)=\mathrm{span}\{\Pi[0]\zeta_0\}\,.
\end{equation}
In particular, the direct sum decomposition 
\begin{equation}\label{sumdec}
\bbWp=\mathrm{span}\{\Pi[0]\zeta_0\} \oplus \mathrm{rg}\big(F_u(\eta_0,0)\big)
\end{equation}
holds.
\end{corollary}

\begin{proof}
As pointed out above, owing to Proposition~\ref{fredholm} and Lemma~\ref{L0a} we only need to check that the eigenvalue 1  of the compact operator $\eta_0 K$ is simple, that is, that $$\mathrm{ker}((1-\eta_0 K)^2)\subset \mathrm{ker}(1-\eta_0 K)\,,$$ since \eqref{sumdec} is implied by Fredholm theory. But if $\psi\in \mathrm{ker}((1-\eta_0 K)^2)$, then 
$$
\phi:= (1-\eta_0 K)\psi\in \mathrm{ker}(1-\eta_0 K) \cap \mathrm{rg}(1-\eta_0 K)\,,
$$
so that $\phi=\alpha \Pi[0]\zeta_0\in \mathrm{rg}(1-\eta_0 K)$ for some $\alpha\in\R$ due to \eqref{kern}. From \eqref{rg} we then deduce that 
$$
\alpha\zeta_0 \in \mathrm{rg}(1-\eta_0Q[0])\cap\mathrm{ker}(1-\eta_0Q[0])=\{0\}\,,
$$ 
the latter equality being due to the fact that $\eta_0^{-1}$ is a simple eigenvalue of the compact operator $Q[0]$. Thus $\alpha=0$ and $\phi=0$, so $\psi\in \mathrm{ker}(1-\eta_0 K)$.
\end{proof}

Based on the foregoing observation we are in a position to apply the theorem of Crandall-Rabinowitz  \cite{CrandallRabinowitz} on local bifurcation.  To get positive solutions it is crucial that $\zeta_0$ is in the interior of the positive cone $E_\varsigma^+$, hence small perturbations of it are still positive. We thus obtain a full branch of nontrivial positive solutions to \eqref{E1aee} in $\mathfrak{S}$ bifurcating from $(\eta,u)=(\eta_0,0)$. We shall use the notation
$$
\bbWppd:= \bbWpp\setminus\{0\}
$$
in the following. The next theorem is from \cite{WalkerJDDE} (see also \cite{WalkerSIMA09}).

\begin{theorem}[Local Bifurcation]\label{locbif}
Suppose   $(A_6)-(A_{10})$ along with $(A_8')$, $(A_{11}')$, and $(A_{12}')$. Let $\eta_0>0$ be given by \eqref{bifvalue}.
Then there are $\ve>0$ and a continuous function $$(\bar{\eta},\bar{u}): [0,\ve)\rightarrow \R\times \bbWp$$ such that the curve 
$$
\mathfrak{K}^+:=\{(\bar{\eta}(t),\bar{u}(t))\,;\, 0\le  t<\ve\} \subset \mathfrak{S}
$$  
bifurcates from the trivial branch $\{(\eta,0)\,;\,\eta\in\R\}$ at $(\bar{\eta}(0),\bar{u}(0))=(\eta_0,0)$ and 
\bqn\label{ps}
\bar{u}(t)=t\, \Pi[0]\zeta_0+o(t)\quad \text{as}\ t\rightarrow 0^+\ .
\eqn 
Near the bifurcation point $(\eta_0,0)$, all nontrivial positive zeros of $F$ lie on the curve $\mathfrak{K}^+$.
Moreover, $$\mathfrak{K}^+\setminus\{(\eta_0,0)\}\subset (0,\infty)\times \bbWppd\,.$$
\end{theorem}

\begin{proof}
Recalling that
$
\mathrm{ker}(F_u(\eta_0,0))=\mathrm{span}\{\Pi[0]\zeta_0\}
$
and that \eqref{FF} implies 
$$
F_{\eta,u}(\eta_0,0) \Pi[0]\zeta_0=-\eta_0^{-1}\Pi[0]\zeta_0\,,
$$
the transversality condition
\begin{equation}\label{transv}
F_{\eta,u}(\eta_0,0)\Pi[0]\zeta_0\not\in\mathrm{rg}(F_u(\eta_0,0))
\end{equation}
is an immediate consequence of \eqref{sumdec}.
 Consequently, owing to \eqref{kern} and \eqref{transv}, we may apply \cite[Theorem 1.7]{CrandallRabinowitz} and deduce the existence of the branch $\mathfrak{K}^+$ as stated in the theorem, where only the positivity of the branch remains to prove. 
For this we note that \eqref{ps} yields 
$$
\frac{1}{t} \gamma_0 \bar{u}(t)=\zeta_0+\gamma_0\frac{o(t)}{t} \quad \text{as}\ t\rightarrow 0\ ,
$$ 
hence, recalling that $\zeta_0\in \mathrm{int}(E_\varsigma^+)$, it follows that  $\gamma_0 \bar{u}(t)\in E_\varsigma^+$ provided $t\in (0,\ve)$ is sufficiently small.
Since 
$$
\bar{u}(t)=\Pi[\bar{u}(t)]\gamma_0 \bar{u}(t)\, ,
$$ 
we conclude $\bar{u}(t)\in \bbWppd$ from $(A_{10})$. This implies the assertion.
\end{proof}

In some cases it is possible to derive more information on the direction of bifurcation. For example, if $Q[v]$ is strongly positive and compact with $r(Q[v])\le 1$ for any $v$, then bifurcation must be supercritical since \eqref{equil} implies that $\eta r(Q[u])=1$ for any (positive) $(\eta,u)\in\mathfrak{S}$, hence~$\eta\ge 1$.\\

The theorem of Crandall $\&$ Rabinowitz  yields an additional curve 
$$
\mathfrak{K}^-:=\{(\bar{\eta}(t),\bar{u}(t))\,;\, \ve<t\le 0\}
$$  
of zeros of $F$ bifurcating from the trivial branch $\{(\eta,0)\,;\,\eta\in\R\}$ at $(\bar{\eta}(0),\bar{u}(0))=(\eta_0,0)$ which is of the form
$$
\bar{u}(t)=t\, \Pi[0]\zeta_0+o(t)\quad \text{as}\ t\rightarrow 0^-\,.
$$
It thus consists of non-positive solutions, i.e. $\mathfrak{K}^-\cap (0,\infty)\times\bbWppd = \emptyset$. 
Near the bifurcation point $(\eta_0,0)$, all nontrivial zeros of $F$ lie on the curve $\mathfrak{K}^-\cup\mathfrak{K}^+$.

\subsubsection{Global Bifurcation}

In this section we make use of the unilateral global bifurcation theory in the spirit of the celebrated alternative of Rabinowitz \cite{Rabinowitz} in order to show that the local curve $\mathfrak{K}^+$ provided by Theorem~\ref{locbif} is contained in a global continuum of positive solutions to \eqref{E1aee}. If the operator $A$ in \eqref{1aee} is independent of $u$, one can directly apply  \cite{Rabinowitz} to derive this result \cite{WalkerJDE10}.
 However, in the quasilinear case, where $A=A(u)$, the function $F$ defined above does not meet the required compactness assumptions from \cite{Rabinowitz}, and we thus rely on a variant of Rabinowitz's alternative proposed in~\cite[Theorem 4.4]{ShiWang} (that, in turn, is based on \cite{PejsachowiczRabier}). \\

In order to apply the results of \cite{ShiWang} we have to strengthen assumption $(A_9)$ to a convexity condition of the form
\Links
\begin{equation}\tag*{${\bf (A_9')}$}
\begin{split} 
&\big(\partial_a+[(1-\alpha) A(0)+\alpha A(u)],\gamma_0\big) \in \mathrm{Isom}\big(\bbWp, L_p(E_0)\times E_{\varsigma}\big)\\
&\text{for each $(\alpha,u)\in [0,1]\times X$}\,.
\end{split}
\end{equation}
\Rechts
This means that $(1-\alpha) A(0)+\alpha A(u)$ possesses maximal $L_p$-regularity, which is satisfied in many applications (see Section~\ref{Sec4}).

\begin{lemma}\label{Fredconvex}
Suppose  $(A_6)-(A_{10})$ along with $(A_8')$, $(A_9')$, $(A_{11}')$, and $(A_{12}')$. Then, given \mbox{$\alpha\in[0,1]$} and $(\eta,u)\in \R\times\bbWp$, the operator $$(1-\alpha)F_u(\eta,0)+\alpha F_u(\eta,u)\in\ml(\bbWp)$$ is Fredholm of index zero.
\end{lemma}

\begin{proof}
Let $\alpha\in[0,1]$ and $(\eta,u)\in \R\times\bbWp$.
Then
\bqnn
\begin{split}
(1-\alpha)F_u(\eta,0)\phi+\alpha F_u(\eta,u)\phi=\, & \phi-S\big(\alpha(\A(0)-\A(u))\phi \,,\, \eta((1-\alpha)\ell[0]+\alpha\ell[u])\phi\big)\\
& -\alpha S\big(-\A_u(u)[\phi] u \,,\, \eta\ell_u[u][\phi] u\big)\,
\end{split}
\eqnn
for $\phi\in\bbWp$. Even though the second operator on the right-hand side of this identity is not compact in general when $\alpha\not= 0$,  one can show, by using assumption $(A_9')$ and computing the kernel and the image as in the proof of Proposition~\ref{fredholm}, that 
$$
\phi \mapsto \phi-S\big(\alpha(\A(0)-\A(u))\phi \,,\, \eta ((1-\alpha)\ell[0]+\alpha\ell[u])\phi\big)
$$
is a Fredholm operator in $\bbWp$ of index zero (see \cite[Proposition 2.2]{WalkerJDDE}). Hence, noticing that
$$
L\phi:= S\big(-\A_u(u)[\phi] u \,,\, \eta\ell_u[u][\phi] u\big)\,,\quad \phi\in\bbWp\,,
$$
defines a compact operator on $\bbWp$ according to assumption $(A_7)$ since $L$ coincides with the Fr\'echet derivative
$$
L= D_w S\big(-\A(w)u \,,\, \eta\ell[w] u\big)\big\vert_{w=u}\in \ml(X,\bbWp)\,,
$$
we conclude the assertion since compact perturbations of Fredholm operators are still Fredholm with the same index.
\end{proof}

We further have to impose that\footnote{Note that $E_0$ is separable since $E_1$ is separable and dense in $E_0$. 
Thus, if $E_0$ is reflexive , then $E_0'$ is separable.} 
\Links
\begin{equation}\tag*{${\bf (A_{15})}$}
\begin{split} 
E_0'\ \text{and } E_1 \text{ are separable .}
\end{split}
\end{equation}
\Rechts
Along with \cite{Restrepo}, this assumption guarantees that 
$\bbWp$ can be equipped with an equivalent norm which is differentiable at any point different from the origin (see \cite[Lemma 3.1]{WalkerJDDE}).

Now we can prove that there is an unbounded continuum of positive solutions to \eqref{E1aee} containing the local branch $\mathfrak{K}^+$  from Theorem~\ref{locbif}. The next theorem is contained in \cite{WalkerJDDE}.

\begin{theorem}[Global Bifurcation]\label{T_continua}
Suppose $(A_6)-(A_{10})$ along with $(A_8')$, $(A_9')$, $(A_{11}')$, $(A_{12}')$ and $(A_{15})$.
Then there is a connected component $\mathfrak{C}^+$ of $\mathfrak{S}$ that is unbounded in $\R\times \bbWp$ and contains the branch~$\mathfrak{K}^+$. Moreover, $$\mathfrak{C}^+\setminus\{(\eta_0,0)\}\subset (0,\infty)\times\bbWppd\,.$$ 
\end{theorem}

\begin{proof}
It readily follows from  \eqref{dhr} and $(A_7)-(A_9)$ that any bounded and closed subset of $\mathfrak{S}$ is compact in $\R\times \bbWp$.
Due to Theorem~\ref{locbif} and Lemma~\ref{Fredconvex}, we may thus apply \cite[Theorem 4.4, Remark 4.2]{ShiWang} and deduce that $\mathfrak{K}^+$ is contained in a connected component  $\mathfrak{C}^+$ of ${\mathfrak{S}}$ and one
of the alternatives
\begin{itemize}
\item[(a)] $\mathfrak{C}^+$ is unbounded in $\R\times\bbWp$, or
\item[(b)] $\mathfrak{C}^+$ contains a point $(\eta_*,0)$ with $\eta_*\not=\eta_0$, or
\item[(c)] $\mathfrak{C}^+$ contains a point $(\eta,z)$ with $z\in \mathrm{rg}\big(F_u(\eta_0,0)\big)\setminus\{0\}$ 
\end{itemize}
occurs, where we have taken into account \eqref{sumdec} in (c).
According to Theorem~\ref{locbif}, the component $\mathfrak{C}^+$ near the bifurcation point $(\eta_0,0)$ coincides with~$\mathfrak{K}^+$. 

Next, we show that $\mathfrak{C}^+\setminus\{(\eta_0,0)\}\subset (0,\infty)\times\bbWppd$. Indeed, if $\mathfrak{C}^+$ leaves $(0,\infty)\times\bbWppd$ at some point $(\eta,u)\in\mathfrak{C}^+\cap \R\times\bbWp$ with $(\eta,u)\notin (0,\infty)\times\bbWppd$, then there is a sequence $((\eta_j,u_j))_{j\in\N}$ in $\mathfrak{C}^+\cap (0,\infty)\times\bbWppd$  such that $$(\eta_j,u_j)\rightarrow (\eta,u)\ \text{ in }\ \R\times\bbWp\,.$$
Clearly, $\eta\ge 0$ and $u\in \bbWpp$ with $\eta=0$ or $u\equiv 0$. But since $(\eta,u)\in\mathfrak{S}$, we readily deduce from \eqref{equil} that $\eta=0$ implies $u\equiv 0$. Hence $u\equiv 0$ in any case, i.e. $(\eta_j,u_j)\rightarrow (\eta,0)$ in $\R\times\bbWp$. Again by \eqref{equil}, we have 
\bqn\label{sol111}
u_j=\Pi[u_j]u_j(0)\ ,\qquad u_j(0)=\eta_j Q[u_j]u_j(0)\ .
\eqn
Since $v_j:=u_j/\|u_j\|_{\bbWp}$ defines a bounded sequence $(v_j)$ in $\bbWp$,  we may use $(A_7)$  and extract a subsequence of $(v_j)$ (which we do not relabel)  which converges to some $v$ in $X$. From \eqref{Pi} and \eqref{Qucont} we deduce $\Pi[u_j]\rightarrow \Pi[0]$ in $\ml(E_\varsigma,\bbWp)$ and $Q[u_j]\rightarrow Q[0]$ in $\ml (E_\varsigma,E_\vartheta)$. 
We then obtain from \eqref{sol111} and \eqref{gamma0} that 
$$
\|u_j(0)\|_{E_\vartheta}\le  \eta_j  \,\|Q[u_j]\|_{\ml (E_\varsigma,E_\vartheta)}\, \|u_j(0)\|_{E_\varsigma}\le c\,\|u_j(0)\|_{E_\varsigma} \le c \,\|u_j\|_{\bbWp}\ ,
$$
and  conclude the boundedness of the sequence $(v_j(0))$ in $E_\vartheta\dhr E_\varsigma$. So, extracting a further subsequence (again not relabeled) we see that $v_j(0)\rightarrow w$ in $E_\varsigma^+$. Letting $j\rightarrow\infty$ in  \eqref{sol111} and using \eqref{Pi} and \eqref{Qucont} yields
$$
v=\Pi[0]w\ ,\qquad w=\eta  Q[0]w\ ,
$$
from which we first deduce that $\eta>0$ since otherwise $w=0$ implying the contradiction \mbox{$v\equiv 0$}. Consequently, $w\in E_\varsigma^+$ is an eigenvector of $Q[0]$ to the eigenvalue $1/\eta$. Thus \mbox{$\eta=\eta_0$} according to Lemma~\ref{L0a}, hence $(\eta,u)=(\eta_0,0)$.
Therefore, $\mathfrak{C}^+$ leaves the set $(0,\infty)\times\bbWppd$ only at $(\eta_0,0)$ and \mbox{$\mathfrak{C}^+\setminus\{(\eta_0,0)\}$} is contained in $ (0,\infty)\times\bbWppd$. In particular, alternative (b) above does not occur. 

We finally show that alternative (c) does not occur as well. Suppose to the contrary that $\mathfrak{C}^+$ contains a point $(\eta,z)$ with $z\not=0$ and $z=F_u(\eta_0,0)\phi$ for some $\phi\in\bbWp$. Then $z\in\bbWppd$ and $\phi-\eta_0\Pi[0]\ell[0]\phi=z$.  Applying the operator $\ell[0]$ on both sides yields 
\begin{equation}\label{z1}
(1-\eta_0 Q[0])\ell[0] \phi =\ell[0] z\,.
\end{equation}
Since $\zeta_0\in\mathrm{int}(E_\varsigma^+)$, we find $\kappa>0$ such that $\psi_0:=\kappa\zeta_0+\ell[0]\phi$ belongs to $E_\varsigma^+$. Combining 
$$
(1-\eta_0 Q[0])\zeta_0 =0
$$
with \eqref{z1} yields $$(1-\eta_0 Q[0])\psi_0 =\ell[0] z$$
contradicting the fact that this equation cannot have a positive solution  $\psi_0\in E_\varsigma^+$ according to Lemma~\ref{L0a} since $r(Q[0])=1/\eta_0$  and $\ell[0] z\in E_\varsigma^+\setminus\{0\}$ due to $(A_{11}')$ and $z\in\bbWppd$.
Therefore, alternative (c) above is impossible and the theorem is proven.
\end{proof}


\section{Examples}\label{Sec4}

We provide some simple applications for the results from the previous section. However, possible applications are not restricted to ones given here, and we also do not strive for providing optimal assumptions.  For other examples we refer to \cite{WalkerSIMA09,WalkerAMPA11,WalkerEJAM} or also to \cite{WalkerCrelle,WalkerNONRWA,WalkerNODEA}, where coexistence in two-species systems systems is studied.

Throughout let $\Omega$ be a bounded and smooth domain in $\R^n$ and $J=[0,a_m]$ with $a_m\in (0,\infty)$.

\subsection{Example I}

Consider a function
\bqn\label{11s}
\beta\in L_\infty(J)\ ,\quad \beta(a)>0\,, \ a\in J\,,
\eqn
that is normalized such that
\bqn\label{12s}
\int_0^{a_m}\beta(a)e^{-\nu_0 a}\,\rd a=1\ ,
\eqn
where $\nu_0>0$ denotes the principal eigenvalue of the negative Laplacian $-\Delta_D$ on $\Om$ subject to homogeneous Dirichlet boundary conditions. We introduce 
$$
\mathbb{W}_p:=L_p(J,W_{p,D}^2(\Omega))\cap W_p^1(J,L_p(\Om))
$$
with $p\in (n+2,\infty)$. Here, $W_{p,D}^s(\Omega)$ for $s\ge 0$ consists of those $v\in W_p^s(\Omega)$ such that $v=0$ on $\partial\Om$ if $s>1/p$. Recall that  $W_{p,D}^s(\Omega)$ embeds compactly in $L_p(\Om)$. We write $\mathbb{W}_p^+$ for the nonnegative functions in $\mathbb{W}_p$ and set $\dot{\mathbb{W}}_p^+:=\mathbb{W}_p^+\setminus\{0\}$. The next result is taken from \cite{WalkerCrelle}. Equation~\ref{13s} below arises when investigating semi-trivial branches in a simple two-species interaction model.

\begin{proposition}\label{P1a}
Let $\alpha >0$ and suppose \eqref{11s} and \eqref{12s}.
For each $\eta>1$ there is a unique solution $u_\eta\in\dot{\mathbb{W}}_p^+$ to
\bqn\label{13s}
\partial_a u-\Delta_D u=-\alpha u^2\ \text{ in }\ J\times\Omega ,\quad u(0,\cdot)=\eta\int_0^{a_m}\beta(a)\, u(a,\cdot)\,\rd a \ \text{ in }\ \Omega\, .
\eqn
The mapping $(\eta\mapsto u_\eta)$ belongs to $C^\infty((1,\infty),\mathbb{W}_p)$ and $\|u_\eta\|_{\mathbb{W}_p}\rightarrow\infty$ as $\eta\rightarrow\infty$. If $\eta\le 1$, then \eqref{13s} has no solution in $\dot{\mathbb{W}}_p^+$.
\end{proposition}

\begin{proof}
First note that 
$$
\big(L_p(\Omega),W_{p,D}^2(\Omega)\big)_{1-1/p,p}\doteq W_{p,D}^{2-2/p}(\Omega) \hookrightarrow C^1(\bar\Omega)
$$
since $p>n+2$, hence $(A_6)$ holds. 
Set $A(u):=-\Delta_D+ \alpha u$. It follows from \cite[Corollary 4]{Simon} and \eqref{BUC}  that
\bqn\label{embb}
\mathbb{W}_p\dhr X:= L_p(J,L_p(\Omega))\ ,\qquad \mathbb{W}_p\hookrightarrow C(J,C^1(\bar{\Om}))\ ,
\eqn
from which we easily deduce that
$$
A\in C^1\big(X,\ml(\mathbb{W}_p,L_p(J,L_p(\Omega))\big)\,,
$$
i.e. $(A_7)$ and $(A_8')$ hold.
Observe then that $A(u)$ has maximal $L_p$-regularity by \cite[III. Example 4.7.3, III. Theorem 4.10.7]{LQPP}, hence $(A_9')$ holds. Moreover, the semigroup $\{e^{a\Delta_D}\,;\, a>0 \}$ consists of strongly positive operators in $\ml (W_{p,D}^{2-2/p}(\Omega))$ by the maximum principle (e.g. \cite[Corollary 13.6]{DanersKochMedina}), hence $(A_{10})$ holds while  \eqref{11s} entails $(A_{12})$ (and in particular $(A_{12}')$). Assumptions $(A_{11}')$ and $(A_{15})$ are clear. Note that
$$
Q[0]=\int_0^{a_m} \beta(a)\,e^{a\Delta_D}\,\rd a
$$
and that, if $\zeta_0$ denotes the positive eigenvector of $-\Delta_D$ to the eigenvalue $\nu_0>0$ with $\|\zeta_0\|_\infty=1$, then $e^{a\Delta_D}\zeta_0=e^{-a\nu_0}\zeta_0$. Thus, \eqref{12s} implies $r(Q[0])=1$. 
 We are therefore in a position to apply Theorem~\ref{T_continua} and conclude the existence of an unbounded continuum $\mathfrak{C}^+$ of solutions in $(0,\infty)\times\dot{\mathbb{W}}_p^+$ emanating from $(1,0)$. 

If $(\eta,u)$ is a solution to \eqref{13s} with $u\in \dot{\mathbb{W}}_p^+$, then
$z'(a)\le-\nu_0z(a)$ for $a\in J$, where
$$
z(a):=\int_\Om \zeta_0\,u(a)\,\rd x\ ,\quad a\in J\ ,
$$
and thus
$$
z(0)=\eta\int_0^{a_m} \beta(a)\int_\Om \zeta_0 u(a)\,\rd x\,\rd a\le \eta\int_0^{a_m}\beta(a)e^{-\nu_0 a}\,\rd a\, z(0)\ .
$$
Since $u\in \dot{\mathbb{W}}_p^+$, this inequality is actually strict. Therefore, we have $z(0)>0$ and so $\eta> 1$ by the above inequality and \eqref{12s}. 

To show uniqueness consider two solutions $u_1, u_2$ to \eqref{13s} in $\dot{\mathbb{W}}_p^+$ for a fixed $\eta>1$. Then 
$$
u_j=\Pi[u_j] u_j(0)\,,\quad \big(1-\eta Q[u_j]\big) u_j(0) =0
$$
and hence $\eta r(Q[u_j])=1$ since $u_j(0)$ is positive (see Lemma~\ref{L0a}). Since $u_1, u_2\in \dot{\mathbb{W}}_p^+$, one can show as in the proof of Corollary~\ref{L1} that $r(Q[u_j])>r(Q[u_1+u_2])$ and thus 
\bqn\label{po}
(1-\eta Q[u_1+u_2])^{-1}\in \ml(W_{p,D}^{2-2/p}(\Omega))\,.
\eqn
On the other hand,
$w:=u_1-u_2$ solves
$$
\partial_a w-\Delta_D w+ \alpha (u_1+u_2) w=0\,,\quad a\in J\,,\qquad w(0)=\eta\int_0^{a_m} \beta(a) w(a)\,\rd a\,,
$$
from which we obtain
$$
\big(1-\eta Q[u_1+u_2]\big) w(0) =0\,.
$$
Hence $w(0)= 0$ by \eqref{po} and then $w=0$. This proves uniqueness.

To continue we recall without proof (see \cite[Lemma 3.6]{WalkerCrelle}) that the maximum principle implies the existence of  a constant $\kappa>0$ such that for any $(\eta,u_\eta)\in\mathfrak{C}^+$ we have
\bqn\label{23ass}
\frac{\nu_0}{\alpha}\frac{\eta-1}{\eta(e^{\nu_0 a}-1)+1-e^{-\nu_0 (a_m-a)}}\,\zeta_0 \,\le\, u_\eta(a)\, \le \,\frac{1}{\alpha a+(\kappa\eta^2)^{-1}} \quad \text{on}\ \,\Om ,\quad a\in J\, .
\eqn
Consider now a sequence
$(\eta_j,u_{\eta_j})\in\mathfrak{C}^+$ with $\|(\eta_j,u_{\eta_j})\|_{\R\times \mathbb{W}_p}\rightarrow\infty$ as $j\rightarrow\infty$. We shall show that $\eta_j \rightarrow\infty$. Suppose otherwise, i.e. let $\eta_j\le\eta_*$ for some $\eta_*>1$. Then necessarily $\|u_{\eta_j}\|_{\mathbb{W}_p}\rightarrow\infty$, and \eqref{23ass} yields
\bqn\label{2ass}
\|u_{\eta_j}(a)\|_\infty\le\kappa\eta_*^2\ ,\quad a\in J\ ,\quad j\in\N\,.
\eqn
The positivity of $u_{\eta_j}$ and \eqref{13s} ensure $0\le u_{\eta_j}(a)\le u_{\eta_j}(0)$ on $\Omega$ for $a\in J$, and thus
$$
\|u_{\eta_j}^2\|_{L_p(J,L_p(\Omega))}^p\,=\, \int_0^{a_m}\int_\Omega(u_{\eta_j}(a))^{2p}\,\rd x\,\rd a\,\le\, a_m\,\|u_{\eta_j}(0)\|_{L_{2p(\Omega)}}^{2p}\ ,\quad j\in\N\, .
$$
Using the property of maximal $L_p$-regularity for $-\Delta_D$, it follows from \eqref{13s} that
\bqn\label{2aas}
\begin{split}
\|u_{\eta_j}\|_{\mathbb{W}_p}\,&\le\, c\,\big(\|\alpha u_{\eta_j}^2\|_{L_p(J,L_p(\Omega))}+\|u_{\eta_j}(0)\|_{W_{p}^{2-2/p}(\Omega)}\big)\\
& \le\,
c\,\big(\|u_{\eta_j}(0)\|_{L_{2p}(\Omega)}^{2}+\|u_{\eta_j}(0)\|_{W_{p}^{2-2/p}(\Omega)}\big)
\end{split}
\eqn
for $j\in\N$ and some constant $c$ independent of $u_{\eta_j}$. Writing the solution $u_{\eta_j}$ to \eqref{13s} in the form
$$
u_{\eta_j}(a)=e^{a\Delta_D}\,u_{\eta_j}(0)- \alpha\int_0^a e^{(a-\sigma)\Delta_D}\, (u_{\eta_j}(\sigma))^2\,\rd \sigma\ ,
$$ 
we see that
$$
u_{\eta_j}(0)=\eta_j\int_0^{a_m}\beta(a)\,
e^{a\Delta_D}\,u_{\eta_j}(0)\,\rd a- \, \alpha\eta_j\int_0^{a_m} \beta(a)\int_0^a e^{(a-\sigma)\Delta_D}\, (u_{\eta_j}(\sigma))^2\,\rd \sigma\,\rd a\ .
$$
Taking into account that $$\|e^{a\Delta_D}\|_{\ml(L_p(\Omega),W_{p,D}^{2-2/p}(\Omega))}\le c a^{1/p-1}\,,\quad a>0\,,$$ (see e.g. \cite{LQPP}), we derive
from \eqref{2ass} that $(u_{\eta_j}(0))_{j\in\N}$ stays bounded in $W_{p}^{2-2/p}(\Omega)$. But then $(u_{\eta_j})_{j\in\N}$ stays bounded in $\mathbb{W}_p^+$ by \eqref{2aas} in contradiction to our assumption. Therefore, $\eta_j \rightarrow\infty$ and, since $\mathfrak{C}^+$ is connected, we conclude that \eqref{13s} admits for each value of $\eta>1$ a unique solution $u_\eta\in\mathbb{W}_p^+$.

We then claim that $\|u_\eta\|_{\mathbb{W}_p}\rightarrow\infty$ as $\eta\rightarrow\infty$. Indeed, assuming to the contrary that $\|u_\eta\|_{\mathbb{W}_p}\le c<\infty$ for all $\eta>1$, we see that $\|u_\eta(0)\|_{\infty}$ is bounded with respect to $\eta$ by \eqref{embb}. Thus $$u_\eta(0)=\eta \int_0^{a_m}\beta(a) u_\eta(a)\rd a$$ implies that $$\left\|\int_0^{a_m}\beta(a) u_\eta(a)\rd a\right\|_\infty \longrightarrow 0 \ \text{ as }\ \eta\longrightarrow \infty\,.$$ However, \eqref{23ass} yields the contradiction 
$$
\frac{\nu_0}{\alpha(1-e^{-\nu_0 a_m})}\,\frac{\eta-1}{\eta}\,\zeta_0\,\le\,\frac{1}{\eta}\, u_\eta(0)\,=\,\int_0^{a_m}\beta(a) u_\eta(a)\rd a\quad \text{on}\ \Om\,.
$$
Consequently, $\|u_\eta\|_{\mathbb{W}_p}\rightarrow\infty$ as $\eta\rightarrow\infty$.

Finally, the statement $(\eta\mapsto u_\eta)\in C^\infty((1,\infty),\mathbb{W}_p)$ is a consequence of the the Implicit Function Theorem.
\end{proof}


\subsection{Example II}

Let the functions 
$$
d:\R\times J\times{\bar\Omega} \rightarrow(\underline{d},\infty)
$$
 and 
$$
\mu, \beta:\R\times J\times \bar{\Omega}\rightarrow\R^+
$$ 
be smooth  with $\underline{d}>0$  and $\beta(0,\cdot,\cdot)> 0$. We then consider the quasilinear problem
\begin{subequations}\label{Aas}
\begin{align}
&\partial_a u-\mathrm{div}_x\big(d(U(x),a,x)\nabla_xu\big)+\mu(U(x),a,x)u=0\,, && a\in (0,a_m)\,, & x\in\Om\,, 
\label{Aa}\\
&u(0,x)=\eta\int_0^{a_m} \beta(U(x),a,x)\, u(a,x)\,\rd a\,, &&& x\in\Om\,, 
\label{B}
\end{align}
\end{subequations}
where $u$ is subject to a boundary condition $$\mathcal{B} u (a,x)=0\,,\quad (a,x)\in (0,a_m)\times\partial\Om\,$$ of either Dirichlet or Neumann type, and where we use the notation
$$
U(x)=\int_0^{a_m}u(a,x)\rd a\,, \quad x\in \Om\, .
$$
Fix $p\in (n+2,\infty)$ and let $\Wpb^s(\Om)$ denote the Sobolev space of all $v\in W_p^s(\Om)$  subject to homogeneous boundary condition $\mathcal{B} v=0$ on $\partial\Omega$ if meaningful (i.e. if $s>1/p$ in the Dirichlet case and $s>1+1/p$ in the Neumann case). 
As in the previous example we introduce
$$
\mathbb{W}_p:=L_p(J,\Wpb^2(\Omega))\cap W_p^1(J,L_p(\Om))
$$
and write $\mathbb{W}_p^+$ for the nonnegative functions in $\mathbb{W}_p$ and $\dot{\mathbb{W}}_p^+:=\mathbb{W}_p^+\setminus\{0\}$.

\begin{proposition}\label{P43}
Given the assumptions above, there is an unbounded continuum of positive solutions $(\eta,u)$ in $(0,\infty)\times\dot{\mathbb{W}}_p^+$ to~\mbox{\eqref{Aas}}
\end{proposition}

\begin{proof}
Standard interpolation theory  \cite[Theorem 4.3.3]{Triebel} yields (for $2\theta\not= 1/p$ or $2\theta\not= 1-1/p$)
\bqn
\label{alpha}
\big(L_p(\Om),\Wpb^2(\Omega)\big)_{\theta,p}\doteq \Wpb^{2\theta}(\Om)
\eqn
and, in particular,
$$
\big(L_p(\Om),\Wpb^2(\Omega)\big)_{1-1/p,p}\doteq \Wpb^{2-2/p}(\Om) \hookrightarrow C^1(\bar{\Om})\,,
$$
hence $(A_6)$ holds.  From \cite[Corollary 4]{Simon} and \eqref{alpha} we obtain for some $\delta>0$ that 
\bqn\label{XX}
\mathbb{W}_p\dhr X:=L_p(J,\Wpb^{2-2/p+\delta}(\Om))
\eqn
and thus $(A_7)$ holds.
In particular, for $u\in X$ we have $$U=\int_0^{a_m}u(a,\cdot)\rd a\in \Wpb^{2-2/p}(\Om)\,.$$ Define, for $u\in X$, $a\in J$,
$$
A(u,a)w:=-\mathrm{div}_x\big(d(U(x),a,x)\nabla_xw\big) +\mu(U(x),a,x) w\ , \quad w\in \Wpb^{2}(\Om)\ ,\quad x\in \Omega\ .
$$
Then, since $d$ and $\mu$ are smooth, we obtain (see e.g. \cite[Proposition 4.1]{WalkerAMPA11}) from \eqref{XX} that
$$
[u\mapsto A(u)]\in C^1\big(X,L_\infty(J,\ml(\Wpb^2(\Om),L_p(\Om)))\big) \hookrightarrow C^1\big(X,\ml(\mathbb{W}_p,L_p(J,L_p(\Om)))\big)
$$
as required by $(A_8')$. Moreover, since $\beta$ is smooth and $\Wpb^{2-2/p+\delta}(\Om)$ is an algebra, we similarly obtain from \eqref{XX} that
$$
[u\mapsto \beta(u)]\in C^1\big(X,\ml(\Wpb^{2-2/p+\delta}(\Om))\big)\,,
$$
hence $(A_{11}')$.
Next note that for $\alpha\in [0,1]$, $u\in X$, and $w\in \Wpb^{2}(\Om)$ we have
 \bqnn
\begin{split}
A_\alpha(u,\cdot)w:&=(1-\alpha)A(0,\cdot)w+\alpha A(u,\cdot)w\\
&=-\mathrm{div}_x\big([(1-\alpha)d(0,\cdot,\cdot)+\alpha d(U,\cdot,\cdot)]\nabla_xw\big) +[(1-\alpha)\mu(0,\cdot,\cdot)+\alpha\mu(U,\cdot,\cdot)] w
\end{split}
\eqnn
with $$(1-\alpha)d(0,\cdot,\cdot)+\alpha d(U,\cdot,\cdot)\ge \underline{d}\ .$$ Hence, for $\alpha\in [0,1]$, $u\in X$, and $a\in J$,  the operator $-A_\alpha(u,a)$ is resolvent positive, generates a contraction semigroup on each $L_q(\Om)$, $1<q<\infty$ (see \cite{AmannIsrael}), and is self-adjoint in $L_2(\Om)$. Hence \cite[III.Example 4.7.3, III.Theorem 4.10.10]{LQPP} entail $(A_9')$. Since, for $u\in X$ fixed, the mapping $$A(u,\cdot):[0,a_m]\rightarrow \ml(\Wpb^2(\Om),L_p(\Om))$$ is H\"older continuous, there is a unique positive evolution operator 
$$
\Pi[u](a,\sigma)\,,\quad 0\le \sigma\le a\le a_m\,,
$$ 
on $E_0$ corresponding to $A(u,\cdot)$, see \cite[II.Corollary 4.4.2, II.Theorem 6.4.2]{LQPP}. In particular,  $(A_{10})$ holds.  Also note that the maximum principle ensures that $\Pi[0](a,0)\in\mathcal{K}(\Wpb^2(\Om))$ is strongly positive for each $a\in J\setminus\{0\}$, see \cite[Section 13]{DanersKochMedina}. Since $\beta(0,\cdot,\cdot)> 0$ we conclude  $(A_{12}')$. Consequently, we are in a position to apply Theorem~\ref{T_continua} and deduce that there is an unbounded continuum of positive solutions $(\eta,u)$ in $(0,\infty)\times\dot{\mathbb{W}}_p^+$ to~\mbox{\eqref{Aas}}.
\end{proof}

Note that Proposition~\ref{P43} holds for more general second-order elliptic operators (e.g. with drift term) and possibly subject to other spatial boundary conditions, e.g. see \cite{WalkerSIMA09}. Also note that in the example above, assumptions $(A_1)-(A_5)$ hold when freezing the nonlinearity $u$ in $A(u,\cdot)$, $\mu(u,\cdot,\cdot)$,  and $\beta(u,\cdot,\cdot)$. In particular, if $a_m<\infty$ and imposing that
\begin{align*}
& d\in C^1([0,a_m],C^1(\bar\Omega))\,, && d(a,x)>0\,,\quad (a,x)\in [0,a_m]\times \bar\Omega\,,\\
&\mu\in C^1([0,a_m],C(\bar\Omega))\,, && \mu(a,x)\ge 0\,,\quad (a,x)\in [0,a_m]\times \bar\Omega\,,\\
&\beta\in C([0,a_m],C^2(\bar\Omega))\,, && \beta(a,x)> 0\,,\quad (a,x)\in [0,a_m]\times \bar\Omega\,,
\end{align*}
then assumptions $(A_1)-(A_5)$ are satisfied in example \eqref{Eu1a} from the introduction.



\begin{thebibliography}{99}




\bibitem{AmannSIAM}
H. Amann. \textit{Fixed point equations and nonlinear eigenvalue problems in ordered Banach spaces.}  SIAM Rev. {\bf 18} (1976), 620-709. 

\bibitem{AmannIsrael}
H. Amann. \textit{Dual semigroups and second-order linear elliptic boundary value problems.} Israel J. Math. {\bf 45} (1983), 225-254.

\bibitem{LQPP}
H. Amann. \textit{Linear and Quasilinear Parabolic Problems. {V}ol. {I}}.
  Monographs in Mathematics, vol.~89, Birkh\"auser Boston Inc., Boston, MA,
  1995.


\bibitem{AnitaAnita}
L.-I. Anita, S. Anita. \textit{Asymptotic behavior of the solutions to semi-linear
age-dependent population dynamics with diffusion
and periodic vital rates.}
Mathematical Population Studies {\bf  15} (2008), 114-121.


\bibitem{14}
 B. Ayati. \textit{A structured-population model of Proteus mirabilis swarm-colony
development.} J. Math. Biol. {\bf 52} (2006), 93--114.

\bibitem{16}
B. Ayati, G. Webb, R. Anderson. \textit{Computational methods and results for structured multiscale models of tumor invasion.} SIAM J. Multsc. Mod. Simul. {\bf 5}
(2006), 1--20.

\bibitem{BatkaiFijavzRhandi}
A. B\'atkai, M.K. Fijav\v{z}, A. Rhandi. \textit{Positive Operator Semigroups}.
  Operator Theory: Advances and Applications {\bf 257}, Birkh\"auser Cham, 2017.

\bibitem{BusenbergLanglais}
S. Busenberg, M. Langlais. \textit{Global behaviour in age structured S.I.S. models with
seasonal periodicities and vertical transmission.} J. Math. Anal. Appl. {\bf 213} (2007), 511-533.

\bibitem{CrandallRabinowitz}
M.G. Crandall, P.H. Rabinowitz. {\it
Bifurcation from simple eigenvalues.}
J. Functional Analysis {\bf 8} (1971), 321-340.

	
\bibitem{CushingI}
J.M. Cushing. \textit{Existence and stability of equilibria in age-structured
population dynamics.} J. Math. Biology {\bf 20} (1984), 259-276.

\bibitem{CushingII}
J.M. Cushing. \textit{Global branches of equilibrium solutions of the McKendrick equations for age structured population growth.} Comp. Math. Appl. {\bf 11} (1985), 175-188.

\bibitem{CushingIII}
J. Cushing. \textit{Equilibria in structured populations.} J. Math. Biol. {\bf 23} (1985), 15-39.	
	
\bibitem{DanersKochMedina}
D. Daners, P. Koch-Medina. \textit{Abstract Evolution Equations, Periodic Problems, and Applications.}
Pitman Res. Notes Math. Ser., {\bf 279}, Longman, Harlow 1992. 
	
	
\bibitem{DelgadoEtAl2}
M. Delgado, M. Molina-Becerra, A. Su\'arez. \textit{Nonlinear age-dependent diffusive equations: A bifurcation approach.}
J. Diff. Equations {\bf 244} (2008), 2133-2155.


\bibitem{DS3}
M. Delgado, M. Molina-Becerra, A. Su\'arez. \textit{A nonlinear age-dependent model with spatial diffusion.}
J. Math. Anal. Appl. {\bf 313} (2006),  366--380.

\bibitem{DS4}
M. Delgado, M. Molina-Becerra, A. Su\'arez. \textit{The sub-supersolution method for an evolutionary reaction-diffusion age-dependent problem.}
Differential Integral Equations {\bf 18} (2005), 155--168. 

\bibitem{DS2}
M. Delgado, A. Su\'arez. 
\textit{Age-dependent diffusive Lotka-Volterra type systems.}
Math. Comput. Modelling {\bf 45} (2007),  668--680. 

\bibitem{Dore}
G. Dore.
\textit{Maximal regularity in $L^p$ spaces for an abstract Cauchy problem.}
Adv. Diff. Equ.  {\bf 5} (2000),  293--322. 

\bibitem{DucroutMagal}
A. Ducrot, P. Magal. \textit{Travelling wave solutions for an infection-age
structured model with diffusion}. Proc. Roy. Soc. Edinburgh Sect. A {\bf 139} (2009), 459--482.

\bibitem{DW2}
J. Dyson, E. Sanchez, R. Villella-Bressan, G.F. Webb. \textit{An age and spatially structured model of tumor invasion with haptotaxis.} Discrete Contin. Dyn. Syst. Ser. B {\bf 8} (2007),  45--60.

\bibitem{DW1}
J. Dyson, R. Villella-Bressan, G.F. Webb. \textit{An age and spatially structured model of tumor invasion with haptotaxis.} II. Math. Popul. Stud. {\bf 15} (2008), 73--95.


\bibitem{EngelNagel}
K.-J. Engel, R. Nagel. \textit{One-Parameter Semigroups for Linear Evolution Equations.}, Springer New York Inc. 2000.

\bibitem{EngelNagel2}
K.-J. Engel, R. Nagel. \textit{A Short Course on Operator Semigroups.} Springer New York Inc. 2006.

\bibitem{ES}
S.E. Esipov and J.A. Shapiro. \textit{Kinetic model of Proteus
mirabilis swarm colony development.} J. Math. Biol. {\bf 36}
(1998), 249--268.

\bibitem{Feller41}
W. Feller. \textit{On the integral equation of renewal theory.} Ann. Math.
Stat. {\bf 12} (1941), 243-267.

\bibitem{54} W. Fitzgibbon, M. Parrott, G. Webb. \textit{Diffusion epidemic models with incubation and crisscross dynamics.} Math. Biosci. {\bf 128} (1995), 131--155.


\bibitem{GuoChan}
B.Z. Guo and W.~L. Chan. \textit{On the semigroup for age dependent population
  dynamics with spatial diffusion}. J. Math. Anal. Appl. \textbf{184} (1994),
  no.~1, 190--199.

\bibitem{Gurtin73}
M.E. Gurtin. \textit{A system of equations for age-dependent population diffusion}. J. Thero. Biol.
{\bf 40} (1973), 389--392.

	
\bibitem{GurtinMacCamy74}
M.E. Gurtin, R.C. MacCamy. \textit{Nonlinear age-dependent population dynamics}. Arch. Rat. Mech. Anal.
{\bf 54} (1974), 281-300.

\bibitem{GurtinMacCamy81}
M.E. Gurtin, R.C. MacCamy. \textit{Diffusion models for age-structured populations}. Math. Biosci.
{\bf 54} (1981), 49--59.


\bibitem{GyllenbergWebb}
M. Gyllenberg, G.~F. Webb. \textit{Asynchronous exponential growth of
semigroups of nonlinear operators}, J. Math. Anal. Appl. \textbf{167} (1992), 443--467.

\bibitem{Huyer}
W. Huyer. \textit{Semigroup formulation and approximation of a linear
  age-dependent population problem with spatial diffusion}, Semigroup Forum
  \textbf{49} (1994), 99--114.

\bibitem{IannelliMartchevaMilner}
M. Iannelli, M. Martcheva, F.A. Milner. \textit{Gender-Structured Population Modeling}. Mathematical Methods, Numerics, and Simulations. SIAM Frontiers in Applied Mathematics, Philadelphia 2005.

\bibitem{Kato1}
T. Kato. \textit{Integration of the equation of evolution in {B}anach spaces}. J. Math. Soc. Japan {\bf 5} (1953), 208-234.

\bibitem{KunyaOizumi}
T. Kunya, R. Oizumi. \textit{Existence result for an age-structured SIS epidemic model
with spatial diffusion}. Nonlinear Anal. Real World Appl.  {\bf 23} (2015), 196--208..

\bibitem{Langlais88}
M. Langlais. \textit{Large time behavior in a nonlinear age-dependent
population dynamics problem with spatial diffusion.} J. Math. Biol. {\bf 26} (1988), 319--346.

\bibitem{LaurencotWalker_Proteus}
Ph.  Lauren\c{c}ot, Ch. Walker. 
\textit{An age and spatially structured population model for Proteus mirabilis swarm-colony development.}
Math. Mod. Nat. Phen. {\bf 7} (2008), 49--77. 

\bibitem{MagalThieme}
P. Magal, H.R. Thieme.
\textit{Eventual compactness for semiflows generated by nonlinear age-structured models.}
 Commun. Pure Appl. Anal. {\bf 3} (2004), 695--727. 



\bibitem{McKendrick}
A.G. McKendrick. \textit{Applications of mathematics to medical problems}. Proc. Edin. Math. Soc.
{\bf 44} (1926), 435-438.

\bibitem{Pazy}
A. Pazy.  \textit{Semigroups of linear operators and applications to partial differential equations}, Applied Mathematical Sciences 44, Springer  1983.

\bibitem{PejsachowiczRabier}
J. Pejsachowicz, P.J. Rabier. \textit{Degree theory for $C^1$ Fredholm mappings of index 0.} J. Anal. Math. {\bf 76} (1998) 289--319.



\bibitem{Pruess83}
J. Pr\"u\ss. \textit{On the qualitative behaviour of populations with age-specific interactions.} Comp. Math. Appl. {\bf 9} (1983), 327-339.


\bibitem{JPQuasi}
J. Pr\"u\ss. \textit{OMaximal regularity for evolution equations in Lp-spaces.} Conf. Sem. Mat. Univ. Bari {\bf 285} (2003), 1-39

\bibitem{Rabinowitz}
P.H. Rabinowitz. \textit{Some global results for nonlinear eigenvalue problems.} J. Funct. Anal. {\bf 7} (1971), 487-513. 

\bibitem{Restrepo}
G. Restrepo. \textit{Differentiable norms in Banach spaces. } Bull. Amer. Math. Soc. {\bf 70}  (1964), 413-414.


\bibitem{Rhandi}
A. Rhandi. \textit{Positivity and stability for a population equation with
  diffusion on {$L^1$}}. Positivity \textbf{2} (1998), no.~2, 101--113.

\bibitem{RhandiSchnaubelt_DCDS99}
A. Rhandi, Roland Schnaubelt. \textit{Asymptotic behaviour of a
  non-autonomous population equation with diffusion in {$L^1$}}. Discrete
  Contin. Dynam. Systems \textbf{5} (1999), no.~3, 663--683.
	
\bibitem{SharpeLotka11}
F.R. Sharpe, A.J. Lotka. \textit{ A problem in age distributions}. Phil. Mag.
{\bf 21} (1911), 98-130.


\bibitem{ShiWang}
J. Shi, X. Wang. \textit{On global bifurcation for quasilinear elliptic systems on bounded domains.}
J. Differential Equations {\bf 246} (2009), no. 7, 2788-2812. 

\bibitem{Simon} J.~Simon. \textsl{Compact sets in the space $L^p(0,T;B)$.} Ann. Mat. Pura Appl. (4) \textbf{146} (1987), 65--96.

\bibitem{ThiemeDCDS}
H.R. Thieme. \textit{Positive perturbation of operator semigroups: growth
  bounds, essential compactness, and asynchronous exponential growth}. Discrete
  Contin. Dynam. Systems \textbf{4} (1998), no.~4, 735--764.
	
\bibitem{ThiemeBook}
H.R. Thieme. \textit{Mathematics in Population Biology.} Princeton Series in Theoretical and Computational Biology. Princeton University Press 2003.

\bibitem {Triebel}
 H. Triebel. 
\textit{Interpolation Theory, Function Spaces, Differential Operators}. 2nd~edition, 
Johann Ambrosius Barth, Heidelberg, 1995.

\bibitem{vFoerster}
H. von Foerster. \textit{Some remarks on changing populations.} The Kinetics of Cellular Proliferation, Grune and Stratton 1959, pp.382-407.

\bibitem{WalkerDIE}
Ch. Walker. \textit{Global Well-Posedness of a Haptotaxis Model Including Age and Spatial Structure.}
Diff. Int. Eq. {\bf 20}  (2007), 1053--1074. 

\bibitem{WalkerEJAM}
Ch. Walker. \textit{Global Existence for an Age and Spatially Structured Haptotaxis Model with Nonlinear Age-Boundary Conditions.}
Europ. J. Appl. Math. {\bf 19} (2008), 113--147.

\bibitem{WalkerSIMA09} Ch. Walker. {\em Positive equilibrium solutions for age and spatially structured population models.} SIAM J. Math. Anal. {\bf 41} (2009), 1366--1387. 

\bibitem{WalkerJDE10}
Ch. Walker. {\em Global bifurcation of positive equilibria in nonlinear population models.}
J. Differential Equations {\bf 248} (2010), 1756--1776.  

\bibitem{WalkerDCDSA10}
Ch. Walker. \textit{Age-dependent equations with non-linear diffusion}. Discrete
  Contin. Dyn. Syst. \textbf{26} (2010), 691--712.

\bibitem{WalkerAMPA11}
Ch. Walker. \textit{Bifurcation of positive equilibria in nonlinear structured
  population models with varying mortality rates}. Ann. Mat. Pura Appl.
  \textbf{190} (2011), 1--19.

\bibitem{WalkerCrelle}
Ch. Walker. \textit{On Positive Solutions of Some System of Reaction-Diffusion Equations with Nonlocal Initial Conditions.}
J. Reine Angew. Math. {\bf 660} (2011), 149--179. 

\bibitem{WalkerNONRWA}
Ch. Walker. \textit{On Nonlocal Parabolic Steady-State Equations of Cooperative or Competing Systems.}
Nonlinear Anal. Real World Appl. {\bf 12} (2011), 3552--3571. 

\bibitem{WalkerNODEA}
Ch. Walker. \textit{Positive Solutions of Some Parabolic System with Cross-Diffusion and Nonlocal Initial Conditions.}
NoDEA Nonlinear Differential Equations and Applications {\bf 19} (2012), 195--218. 


\bibitem{WalkerMOFM}
Ch. Walker. {\em Some Remarks on the Asymptotic Behavior of the Semigroup Associated with Age-Structured Diffusive Populations.}
Monatsh. Math. {\bf 170} (2013), 481--501. 

 

\bibitem{WalkerJDDE}
Ch. Walker. \textit{Global Continua of Positive Solutions for Some Quasilinear Parabolic Equations with a Nonlocal Initial Condition.}
J. Dynam. Differential Equations {\bf 25} (2013), 159--172. 



\bibitem{WebbBook}
 G.F. Webb. \textit{Theory of nonlinear age-dependent population dynamics. Monographs and Textbooks in Pure and Applied Mathematics},
{\bf 89} (1985), Marcel Dekker, Inc., New York.

\bibitem{120}
 G.F. Webb.  \textit{An age-dependent epidemic model with spatial diffusion.} Arch. Rat.
Mech. Anal. {\bf 75} (1980), 91--102.

\bibitem{Webb_TAMS_87}
G.~F. Webb. \textit{An operator-theoretic formulation of asynchronous exponential
  growth}. Trans. Amer. Math. Soc. \textbf{303} (1987), no.~2, 751--763.

\bibitem{WebbSpringer}
G.F. Webb. {\em Population models structured by age, size, and
spatial position.} Structured population models in biology and epidemiology, 1–49,
Lecture Notes in Math. {\bf 1936}, Math. Biosci. Subser., Springer, Berlin, 2008. 



\end{thebibliography}
\end{document}